\documentclass[11pt]{amsart}
\usepackage{amsmath,amsthm, amscd, amssymb, amsfonts, mathrsfs}
\usepackage[all]{xy}
\usepackage{color}

\usepackage{calligra}
\usepackage[T1]{fontenc}

\newcommand{\se}{\mathfrak S}

\newcommand{\com}[1]{\textcolor{blue}{\textbf{#1}}}

\newcommand{\Ga}{\varGamma}
\newcommand{\La}{\varLambda}

\newcommand\Supp{\operatorname{Supp}}
\newcommand\Soc{\operatorname{Soc}}

\newcommand\add{\operatorname{add}}

\newcommand{\cS}{\mathcal{S}}

\newcommand{\cA}{\mathcal{A}}
\newcommand{\toba}{\mathcal{B}}
\newcommand{\wtoba}{\widehat{\mathcal{B}}}
\newcommand{\C}{{\mathcal C}}
\newcommand{\D}{{\mathcal D}}
\newcommand{\Ee}{{\mathcal E}}
\newcommand{\F}{{\mathcal F}}

\newcommand{\cJ}{\mathcal{J}}
\newcommand{\J}{{\mathcal J}}

\newcommand{\M}{{\mathcal M}}

\newcommand{\oc}{{\mathcal O}}
\newcommand{\Oc}{{\mathcal O}}
\newcommand{\Pc}{{\mathcal P}}

\newcommand{\R}{\mathbb R}
\newcommand{\Rc}{{\mathcal R}}

\newcommand{\Ss}{{\mathcal S}}
\def\Tc{\mathcal{T}}

\newcommand\Udcp{\mathrm{U}}
\newcommand\Ulus{\mathcal{U}}

\newcommand{\lus}{{\mathcal L}}

\newtheorem{lema}{Lemma}[section]
\newtheorem{prop}[lema]{Proposition}

\newtheorem{Thm}[lema]{Theorem}

\theoremstyle{definition}
\newtheorem{Def}[lema]{Definition}
\newtheorem{definition}[lema]{Definition}
\newtheorem{Rem}[lema]{Remark}

\newtheorem{Ex}[lema]{Example}
\newtheorem{Exs}[lema]{Examples}

\newtheorem{claim}{Claim}
\newtheorem{step}{Step}

\usepackage{hhline}

\newcommand{\ub}{{\mathbf u}}
\newcommand{\vb}{{\mathbf v}}

\newcommand{\fd}{finite-dimen\-sional}
\newcommand{\Fd}{Finite-dimensional}

\newcommand{\ba}{\mathbf{a}}

\renewcommand{\_}[1]{_{\left( #1 \right)}}
\renewcommand{\^}[1]{^{\left( #1 \right)}}

\newcommand{\cou}{\varepsilon }
\newcommand{\ot}{{\otimes}}

\newcommand{\Sn}{{\mathbb S}}

\newcommand{\xij}[1]{x_{(#1)}}

\newcommand\g{\mathfrak{g}}

\newcommand{\ku}{\Bbbk}
\newcommand{\Z}{{\mathbb Z}}
\newcommand{\N}{{\mathbb N}}
\newcommand{\I}{{\mathbb I}}

\newcommand\gA{\mathfrak{A}}

\newcommand{\Co}{\underline{\C}}

\newcommand{\uno}{{\bf 1}}

\newcommand{\ydh}{{}^H_H\mathcal{YD}}

\newcommand{\qdim}{\operatorname{qdim}}
\newcommand{\Imm}{\operatorname{Im}}

\newcommand{\Vect}{\operatorname{Vec}}

\newcommand{\End}{\operatorname{End}}
\newcommand{\Aut}{\operatorname{Aut}}
\newcommand{\Ext}{\operatorname{Ext}}
\newcommand{\Ind}{\operatorname{Ind}}
\newcommand{\Pro}{\operatorname{Pro}}
\newcommand\Res{\operatorname{Res}}

\newcommand\tr{\operatorname{tr}}
\newcommand\Rep{\operatorname{Rep}}
\newcommand\Corep{\operatorname{Corep}}
\newcommand\Repo{\operatorname{\underline{Rep}}}
\newcommand\Corepo{\operatorname{\underline{Corep}}}

\newcommand\sgn{\operatorname{sgn}}
\newcommand\ad{\operatorname{ad}}
\newcommand\Hom{\operatorname{Hom}}

\newcommand\id{\operatorname{id}}

\def\pf{\begin{proof}}
\def\epf{\end{proof}}

\numberwithin{equation}{section}\theoremstyle{plain}

\newcommand\Alg{\operatorname{Alg}}

\newcommand\Ker{\operatorname{Ker}}
\newcommand\Rad{\operatorname{Rad}}

\newcommand\ord{\operatorname{ord}}

\newcommand\Irr{\operatorname{Irr}}
\newcommand\Indesc{\operatorname{Indec}}

\newcommand\md{\operatorname{Mod}}
\newcommand\mdf{\operatorname{mod}}

\newcommand\lm[1]{{{#1}{\text -}\md}}
\newcommand\lmf[1]{{{#1} {\text -}\mdf}}

\newcommand\YDg{^{\ku G}_{\ku G}\mathcal{YD}}

\newcommand\YDgd{^{\ku^G}_{\ku^G}\mathcal{YD}}

\newcommand\YDH{{}^{H}_{H}\mathcal{YD}}

\newcommand\wH{\widetilde H}

\newcommand\U{{\mathfrak {u}}}
\newcommand\B{\U^{\mathfrak {b}}}

\newcommand{\ideal}{\widehat{{\mathcal J}}}
\newcommand{\Ig}{{\mathfrak I}}

\begin{document}

%\renewcommand{\baselinestretch}{1.2}

%\thispagestyle{empty}
%\vspace*{2in}

\title[From Hopf algebras to tensor categories]{From Hopf algebras to tensor categories}
\author[Andruskiewitsch, Angiono,  Garc\'ia Iglesias, Torrecillas, Vay]
{N. Andruskiewitsch, I. Angiono, A. Garc\'ia Iglesias,\\ B. Torrecillas, C. Vay}

\address{N. A., I. A., A. G. I., C. V.: FaMAF-CIEM (CONICET), Universidad Nacional de C\'ordoba,
Medina A\-llen\-de s/n, Ciudad Universitaria, 5000 C\' ordoba, Rep\'
ublica Argentina.} \email{(andrus|angiono|aigarcia|vay)@famaf.unc.edu.ar}

\address{B. T.: Universidad de Almer\'\i a, Dpto. \'Algebra y An\'alisis Matem\'atico.
E04120 Almer\'\i a, Spain}

\email{btorreci@ual.es}

\thanks{\noindent 2000 \emph{Mathematics Subject Classification.}
16W30. \newline The work of N. A., I. A., A. G. I. and C. V. was partially supported by CONICET,
FONCyT-ANPCyT and Secyt (UNC). The work of B. T. was was partially supported by MTM 2011-27090, Ministerio de Educaci\'on y Ciencia (Espa\~na) and the project FQM 3128 from Junta Andaluc\'\i a (FEDER)}

\begin{abstract}
This is a survey on spherical Hopf algebras.
We give criteria to decide when a Hopf algebra is spherical and collect examples. We discuss tilting modules as a mean to obtain a fusion subcategory of the non-degenerate quotient of the category of representations of a suitable Hopf algebra.
\end{abstract}

\maketitle

\section*{Introduction}\label{sect:intro}

It follows from its very definition that the category $\Rep H$ of \fd{} representations of a Hopf algebra $H$ is a tensor category.
There is a less obvious way to go from Hopf algebras with some extra structure (called spherical Hopf algebras) to tensor categories.
Spherical Hopf algebras and the procedure to obtain a tensor category from them were introduced by Barrett and Westbury \cite{BaW-adv, BaW-tams},
inspired by previous work by Reshetikhin and Turaev \cite{RT-cmp, RT-inv}, in turn motivated to give a mathematical foundation to the work of Witten \cite{Wi}.
A spherical Hopf algebra has by definition a group-like element that implements the square of the antipode (called a pivot) and satisfies the left-right trace
symmetry condition \eqref{eq:omega-traza}. The classification (or even the characterization)  of spherical Hopf algebras is far from being understood,
but there are two classes to start with. Let us first observe that semisimple spherical Hopf algebras are excluded from our considerations,
since the tensor categories arising from the procedure are identical to the categories of representations. Another remark: any Hopf algebra
is embedded in a pivotal one, so that the trace condition \eqref{eq:omega-traza} is really the crucial point. Now the two classes we mean are

\begin{itemize}
  \item Hopf algebras with involutory pivot, or what is more or less the same, with $\Ss^4 = \id$. Here the trace condition follows for free, and the quantum dimensions will be (positive and negative) integers.

  \item Ribbon Hopf algebras \cite{RT-cmp, RT-inv}.
\end{itemize}

It is easy to characterize pointed or copointed Hopf algebras with $\Ss^4 = \id$; so we have many examples of (pointed or copointed)
spherical Hopf algebras with involutory pivot, most of them not even quasi-triangular, see Section \ref{sec:inv-pivot}.
On the other hand, any quasitriangular Hopf algebra is embedded in a ribbon one \cite{RT-cmp}; combined with the construction
of the Drinfeld double, we see that any \fd{} Hopf algebra gives rise to a ribbon one. So, we have plenty of examples of spherical Hopf algebras.

The procedure to get a tensor category from a spherical Hopf algebra $H$ consists in taking a suitable quotient $\Repo H$ of the category $\Rep H$. This appears in \cite{BaW-adv} but similar ideas can be found elsewhere, see e.~g. \cite{GK-inv, Kn}. The resulting spherical categories are semisimple but seldom have a finite number of irreducibles, that is, they are seldom fusion categories in the sense of \cite{ENO}. We are interested in describing fusion tensor subcategories of $\Repo H$ for suitable $H$. This turns out to be a tricky problem. First, if the pivot is involutive, then the fusion subcategories of $\Repo H$ are integral, see Proposition \ref{prop:spherical-involutory}.
The only way we know is through tilting modules; but it seems to us that there is no general method, just a clever recipe that works.
This procedure has a significant outcome in the case of quantum groups at roots of one, where the celebrated Verlinde categories are obtained \cite{AP}; see also \cite{Sa} for a self-contained exposition and \cite{Mat} for similar results in the setting of algebraic groups over fields of positive characteristic. One should also mention that the Verlinde categories can be also constructed from vertex operator algebras related to affine Kac-moody algebras, see \cite{BK, Hu, HuL, KaLu} and references therein; the comparison of these two approaches is highly non-trivial. Another approach, at least for $SL(n)$, was proposed in \cite{hayashi} via face algebras (a notion predecessor of weak Hopf algebras).

The paper is organized as follows. Section \ref{sec:preliminaries} contains some information about the structure of Hopf algebras and notation used later in the paper. Section \ref{sec:sph-hopf} is devoted to spherical Hopf algebras. In Section \ref{sec:titling2} we discuss tilting modules and how this recipe would work for some \fd{} pointed Hopf algebras associated to Nichols algebras of diagonal type, that might be thought of as generalizations of the small quantum groups of Lusztig.

\section{Preliminaries}\label{sec:preliminaries}

\subsection{Notations}
Let $\ku$ be an  algebraically closed field of characteristic 0 and $\ku^{\times}$ its multiplicative group of units.
All vector spaces, algebras, unadorned Hom and $\ot$ are over $\ku$.
By convention, $\N = \{1, 2, \dots\}$ and $\N_0 = \N \cup \{0\}$.

Let $G$ be a finite group. We denote by $Z(G)$
the center of $G$ and by $\Irr G$ the set of isomorphism classes
of irreducible representations of $G$. If $g\in G$, we denote
by $C_{G}(g)$ the centralizer of $g$ in $G$. The conjugacy class
of $g$ is denoted by $\oc_g$ or by $\oc^G_g$, when emphasis on the group
is needed.

Let $A$ be an algebra. We denote by $Z(A)$
the center of $A$. The category of $A$-modules is denoted $\lm{A}$; the full subcategory of \fd{} objects is denoted $\lmf{A}$.
The set of isomorphism classes of irreducible\footnote{that is, minimal, see page \pageref{page:minimal}.} objects in an
abelian category $\C$ is denoted $\Irr \C$; we use the abbreviation $\Irr A$ instead of $\Irr \lmf{A}$.

\subsection{Tensor categories}
We refer to \cite{BK, McL, EGNO, ENO, EO} for basic results and terminology on tensor and monoidal categories.
A monoidal category is one with tensor product and unit, denoted $\uno$; thus $\End (\uno)$ is a monoid.
A monoidal category is rigid when it has right and left dualities. In this article, we understand by
tensor category a monoidal rigid abelian $\ku$-linear category, with $\End (\uno)\simeq \ku$.
A particular important class of tensor categories is that of fusion categories, that is semisimple tensor categories with finite set of isomorphism classes of simple objects, that includes the unit object, and \fd{} spaces of morphisms.
Another important class of tensor categories is that of braided tensor
categories, i.~e. those with a commutativity constraint $c_{V,W}: V \otimes W \to W \ot V$ for every objects $V$ and $W$.
A braided vector space is a pair $(V, c)$ where $V$ is a vector space and $c\in GL(V\ot V)$ satisfies $(c\ot\id)(\id\ot c)(c\ot\id) = (\id\ot c)(c\ot\id)(\id\ot c)$;
this notion is closely related to that of braided tensor category.

\subsection{Hopf algebras}\label{subsec:hopf}

We use standard notation for Hopf algebras (always assumed with bijective antipode); $\Delta$,
$\cou$,  $\Ss$, denote respectively the comultiplication, the counit, and the antipode. For the first, we use
the Heyneman-Sweedler notation $\Delta(x) = x\_{1}\otimes x\_{2}$.
The tensor category of \fd{} representations of a Hopf algebra $H$ is denoted $\Rep H$ instead of $\lmf{H}$, to stress the tensor structure. There are two duality endo-functors of $\Rep H$ composing the transpose of the action with either the antipode or its inverse: if $M \in \Rep H$, then $M^* = \Hom(M, \ku) = {}^*M$, with actions
\begin{align*}
\langle h\cdot f, m\rangle &= \langle f, \Ss(h)\cdot m\rangle, &
\langle h\cdot g, m\rangle &= \langle g, \Ss^{-1}(h)\cdot m\rangle,
\end{align*}
for $h\in H$, $f\in M^*$, $g\in {}^*M$, $m\in M$.

\smallbreak The tensor category of \fd{} corepresentations of a Hopf algebra $H$, i.~e. right comodules, is denoted $\Corep H$.
The coaction map of $V\in \Corep H$ is denoted $\rho =\rho_V:V\to V\otimes H$; in  Heyneman-Sweedler notation, $\rho(v) = v\_{0}\otimes v\_{1}$, $v\in V$.
Also, the coaction map of a left comodule $W$ is denoted $\delta =\delta_W: W\to W\otimes H$; that is, $\delta(w) = w\_{-1}\otimes w\_{0}$, $w\in W$.
%;the space of $H$-comodule maps between $V$ and $W\in \Corep H$ is denoted $\Hom^H(V, W)$.

A Yetter-Drinfeld module $V$ over a Hopf algebra $H$ is simultaneously a left $H$-module and a left $H$-comodule, subject to the
compatibility condition $\delta(h\cdot v) = h\_1 v\_{-1}\Ss(h\_3) \ot h\_2\cdot v\_0$ for $v\in V$, $h\in H$. The category $\ydh$ of Yetter-Drinfeld modules over $H$ is a braided tensor category with braiding $c(v\ot w) = v\_{-1}\cdot w \ot v\_0$, see e.~g. \cite{AS-cambr,Mo}; when $\dim H < \infty$, $\ydh$ coincides with the category of representations of the Drinfeld double $D(H)$.

\smallbreak Let $H$ be a Hopf algebra. A basic list of $H$-invariants is
\begin{itemize}
  \item The group $G(H)$ of group-like elements of $H$,
  \item the coradical $H_0 =$ largest cosemisimple subcoalgebra of $H$,
  \item the coradical filtration of $H$.
\end{itemize}

Assume that $H_0$ is a Hopf subalgebra of $H$.
In this case, another fundamental invariant of $H$ is the \emph{infinitesimal braiding}, a Yetter-Drinfeld module $V$ over $H_0$, see \cite{AS-cambr}.
We shall consider two particular cases:

\begin{itemize}
  \item
The Hopf algebra $H$ is \emph{pointed} if $H_0 = \ku G(H)$.

  \item
The Hopf algebra $H$ is \emph{copointed} if $H_0 = \ku^G$ for a finite group $G$.
\end{itemize}

Recall that $x\in H$ is $(g,1)$ skew-primitive when $\Delta(x) = x\ot 1 + g\ot x$; necessarily, $g\in G(H)$. If $H$ is generated as an algebra by group-like and skew-primitive elements, then it is pointed.

We shall need the description of \emph{Yetter-Drinfeld modules} over $\ku G$, $G$ a finite group; these are $G$-graded vector
spaces $M = \oplus_{g\in G} M_g$ provided with a $G$-module structure such that
$g\cdot M_t = M_{gtg^{-1}}$ for any $g,t\in G$. The category $\YDg$ of
Yetter-Drinfeld modules over $G$ is semisimple and its irreducible objects
are parameterized by pairs $(\oc, \rho)$, where $\oc$ is a conjugacy class of $G$ and $\rho\in \Irr C_{G}(g)$, $g\in \oc$
fixed. We describe the corresponding irreducible
Yetter-Drinfeld module $M(\oc, \rho)$. Let $g_1 = g$, \dots, $g_{m}$ be a
numeration of $\oc$ and let $x_i\in G$ such that $x_i \, g\, x_i^{-1} = g_i$ for all
$1\le i \le m$. Then
\[
M(\oc, \rho) =\operatorname{Ind}^G_{C_G(g)}V=\oplus_{1\le i \le m} x_i\otimes V.
\]
Let $x_iv := x_i\otimes v \in M(\oc,\rho)$, $1\le i \le m$, $v\in V$.
The Yetter-Drinfeld module $M(\oc,\rho)$
is a braided vector space with braiding given by
\begin{equation}
\label{yd-braiding}
c(x_iv\otimes x_jw) = g_i\cdot(x_jw)\otimes x_iv =
x_h\,\rho(\gamma)(w) \otimes x_iv
\end{equation}
for any $1\le i,j\le m$, $v,w\in V$, where $g_ix_j = x_h\gamma$ for unique $h$,
$1\le h \le m$, and $\gamma \in C_{G}(g)$. Now the categories $\YDg$ and $\YDgd$ are tensor equivalent, so that a
similar description of the objects of the latter holds, see e.~g. \cite{AV}.

\medbreak
The following notion is appropriate to describe
all braided vector spaces arising as Yetter-Drinfeld modules over some finite abelian group.
A braided vector space $(V,c)$ is of \emph{diagonal type} if there
exist $q_{ij}\in \ku^{\times}$ and a basis $\{x_i\}_{i\in I}$ of $V$
such that $c(x_i \ot x_j)= q _{ij} x_j \ot x_i$, for each pair
$i,j\in I$. In such case, we say that $i,j\in I$ are \emph{connected} if
there exist $i_k\in I$, $k=0,1,\ldots,n$, such that $i_0=i$,
$i_n=j$, and $q_{i_{k-1},i_k}q_{i_k,i_{k-1}}\neq 1$, $1\le k \le n$. It establishes
an equivalence relation on $I$. The equivalence classes are called
connected components, and $V$ is \emph{connected} if it has a unique
component.

\medbreak When $H=\ku \Gamma$, where $\Gamma$ is a finite abelian group, each $V\in \YDH$ is a
braided vector space of diagonal type. Indeed, $V = \oplus_{g\in
\Gamma, \chi\in \widehat{\Gamma}}V_{g}^{\chi}$, where  $V_{g} = \{v\in V \mid \delta(v) = g\otimes
v\}$, $V^{\chi} = \{v\in V \mid  g \cdot v = \chi(g)v \text{ for all
} g \in \Gamma\}$, $V_{g}^{\chi}
= V^{\chi} \cap V_{g}$. Note that $ c(x\otimes y) =\chi(g) y\otimes x$, for each $x\in V_{g}$, $g \in \Gamma$, $y\in
V^{\chi}$, $\chi \in \widehat{\Gamma}$.
On the other hand, we can realize every braided vector space of diagonal type as a Yetter-Drinfeld module over the group algebra of an
appropriate abelian group.

\subsection{Nichols algebras}\label{subsec:nichols}

Let $H$ be a Hopf algebra.
We shall say \emph{braided} Hopf algebra for a Hopf algebra in the braided tensor category $\ydh$.
Given $V\in \YDH$, the \emph{Nichols algebra} of $V$ is the braided
graded Hopf algebra $\toba(V)=\oplus_{n\geq 0}\toba^n(V)$ satisfying
the following conditions:
\begin{itemize}
  \item $\toba^0(V)\simeq\ku$, $\toba^1(V)\simeq V$ as
  Yetter-Drinfeld modules over $H$;
  \item $\toba^1(V)=\Pc\big(\toba(V)\big)$, the set of primitives elements of
  $\toba(V)$;
  \item $\toba(V)$ is generated as an algebra by $\toba^1(V)$.
\end{itemize}
The Nichols algebra of $V$ exists and is unique up to isomorphism.
We sketch a way to construct $\toba(V)$ and prove its
unicity, see \cite{AS-cambr}. Note that the tensor algebra $T(V)$
admits a unique structure of graded braided Hopf algebra in $\YDH$
such that $V \subseteq \Pc(V)$. Consider the class $\se$ of all the
homogeneous two-sided Hopf ideals $I \subseteq T(V)$ such that $I$
is generated by homogeneous elements of degree $\geq 2$ and is a
Yetter-Drinfeld submodule of $T(V)$. Then $\toba(V)$ is the
quotient of $T(V)$ by the maximal element $I(V)$ of $\se$. Thus, the canonical projection $\pi:T(V)\to \toba(V)$ is a
Hopf algebra surjection in $\YDH$.

\smallbreak
Braided vector spaces of diagonal type with \fd{} Nichols algebra were classified in \cite{He-adv}; the explicit defining relations were given in \cite{Ang3}, using the results of \cite{Ang2}.
Presently we understand that the list of braidings of diagonal type with \fd{} Nichols algebra given in \cite{He-adv} is divided into three parts:

\begin{enumerate}
	\item Standard type \cite{Ang1}, comprising Cartan type \cite{AS-adv};
	
	\item Super type \cite{AAY};
	
	\item Unidentified type \cite{Ang4}.
\end{enumerate}

\section{Spherical Hopf algebras}\label{sec:sph-hopf}
\subsection{Spherical Hopf algebras} The notion of \emph{spherical Hopf algebra} was introduced in \cite{BaW-adv}: this
is a pair $(H, \omega)$, where $H$ is a Hopf algebra and  $\omega\in G(H)$ such that
\begin{align}\label{eq:omega-square-antipode}
\Ss^2(x) &= \omega x \omega^{-1}, & x&\in H, \\
\label{eq:omega-traza}
\tr_V(\vartheta \omega) &= \tr_V(\vartheta \omega^{-1}),  & \vartheta &\in \End_H(V), &&
\end{align}
for all $V\in \Rep H$.
We  say that $\omega\in G(H)$ is a \emph{pivot} when it satisfies \eqref{eq:omega-square-antipode}; pairs $(H, \omega)$ with $\omega$ a pivot are called \emph{pivotal Hopf algebras}. The pivot is not unique but
it is determined up to multiplication by an element in the group $G(H) \cap Z(H)$.
A \emph{spherical element} is a pivot that fulfills \eqref{eq:omega-traza}.

The implementation of the square of the antipode by conjugation by a group-like, condition \eqref{eq:omega-square-antipode}, is easy to verify. For instance, $\omega$ should belong to the center of $G(H)$; thus, if this group is centerless, then \eqref{eq:omega-square-antipode} does not hold in $H$.
Further, the failure of \eqref{eq:omega-square-antipode} is not difficult to remedy by adjoining a group-like element \cite[Section 2]{So}. Namely, given a Hopf algebra $H$, consider a cyclic group $\Gamma$  of order $\ord \Ss^2$ with a generator $g$; let $g$ act on $H$ as $\Ss^2$. Then the corresponding smash product $E(H) := H\# \ku \Gamma$  is a Hopf algebra where \eqref{eq:omega-square-antipode} holds\footnote{If $H$ has finite dimension $n\in \N$, then the order of $\Ss^2$ is finite; in fact, it divides $2n$, by Radford's formula on $\Ss^4$ and the Nichols-Z\"oller Theorem.}.

Condition \eqref{eq:omega-traza} is less apparent. If $H$ is a \fd{} Hopf algebra and
$\omega\in H$ a pivot, then \eqref{eq:omega-traza} holds in the following instances:
\begin{itemize}
  \item $\omega$ is an involution.

  \item There exists a Hopf subalgebra $K$ of $H$ such that $\omega \in K$ and $(K, \omega)$ is spherical-- since $\End_H(V) \subset \End_K(V)$.

  \item $H$ is ribbon, see Subsection \ref{subsec:ribbon}.

  \item All \fd{} $H$-modules are naturally self-dual.

\end{itemize}

\begin{proof}
 By hypothesis, there exists a natural isomorphism $F:\id\rightarrow {}^*$. Let $V\in\Rep H$ and $\vartheta\in\End_H(V)$. Then
\begin{align*}
\tr(\vartheta\omega)&=\sum_i\langle\alpha_i,\vartheta\omega v_i\rangle=\sum_i\langle\vartheta^*\alpha_i,\omega v_i\rangle=\sum_i\langle F_V\vartheta F_V^{-1}\alpha_i,\omega v_i\rangle\\
&=\sum_i\langle F_V\vartheta\omega^{-1} F_V^{-1}\alpha_i, v_i\rangle=\tr(F_V\vartheta\omega^{-1} F_V^{-1})=\tr(\vartheta\omega^{-1}).
\end{align*}
Here $\{v_i\}$ and $\{\alpha_i\}$ are dual basis of $V$ and $V^*$ respectively.
\end{proof}

\begin{prop}\label{prop:reduccion a irreducibles}
A pivotal Hopf algebra $(H,\omega)$ is spherical if and only if
\eqref{eq:omega-traza} holds for all $S\in \Irr H$.
\end{prop}
\pf The Proposition follows from the following two claims.
\begin{claim}
If  \eqref{eq:omega-traza} holds for $M_1,M_2\in \Rep H$, then it holds for $M_1\oplus M_2$.
\end{claim}
Indeed, let $h\in \End_H(M_1\oplus M_2)$. Let $\pi_j:M\to M_j$, $\iota_i:M_i\to
M$ be the projection and the inclusion, for $1\leq i,j\leq 2$. Then $h=\sum_{1\leq i,j\leq
2}h_{ij}$, where $h_{ij}=\pi_i\circ h\circ \iota_j\in \Hom_H(M_j,M_i)$. In particular,
$h_{ij}\omega=(h\omega)_{ij}$ as linear maps. Now, we have
that $\tr_M(h)=\tr_{M_1}(h_{11})+\tr_{M_2}(h_{22})$ and thus
\begin{align*}
 \tr_M(h\omega)&=\tr_{M_1}(h\omega)_{11}+\tr_{M_2}(h\omega)_{22}=\tr_{M_1}(h_{11}\omega)
+\tr_{M_2}(h_{22}\omega)\\
&=\tr_{M_1}(h_{11}\omega^{-1})+\tr_{M_2}(h_{22}\omega^{-1})=\tr_{M_1
}(h\omega^{-1})_{11}+\tr_{M_2}(h\omega^{-1})_{22}\\
&=\tr_M(h\omega^{-1}).
\end{align*}
\begin{claim}
If \eqref{eq:omega-traza} holds for every semisimple $H$-module, then $H$ is spherical.
\end{claim}
Let $M\in\Rep H$ and let $M_0\subset M_1\subset\dots\subset
M_k=M$ be the {\it Loewy filtration} of $M$, that is $M_0=\Soc M$,
$M_{i+1}/M_i \simeq \Soc(M/M_i)$, $i=0,\dots,
j-1$. In particular, $M_{i+1}/M_i$ is semisimple. We prove the claim by induction on the
Loewy length $k$ of $M$. The case $k=0$ is the hypothesis. Assume $k>0$; set
$S=\Soc M$ and consider the exact sequence $0\to S\to M\to M/S\to 0$. Hence
the Loewy length of $\widetilde M=M/S$ is $k-1$ and thus \eqref{eq:omega-traza} holds
for it. Also, \eqref{eq:omega-traza} holds for $S$ by hypothesis. Let $f\in
\End_H(M)$, then $f(S)\subseteq S$ and thus $f$ induces $f_1=f_{|S}\in\End_H(S)$ by
restriction and factorizes through $f_2\in \End_H(\widetilde M)$. Therefore, we can
choose a basis of $S$ and complete it to a basis of $M$ in such a way that $f$ in this new
basis is represented by $[f]=\left[\begin{smallmatrix}
                                   [f_1]&*\\
                                   0&[f_2]
                                  \end{smallmatrix}\right]$. Also, as $\omega$ preserves
$S$ and $f$ is an $H$-morphism, it follows that $[f\omega]=\left[\begin{smallmatrix}
                                   [f_1\omega]&*\\
                                   0&[f_2\omega]
                                  \end{smallmatrix}\right]$ and thus
$\tr_M(f\omega)=\tr_M(f\omega^{-1})$.
\epf

\begin{Ex}
Let $H$ be a basic Hopf algebra, i.~e. all \fd{} simple modules have dimension 1;
when $H$ itself is \fd{}, this amounts to the dual of $H$ being pointed.
If $\omega\in G(H)$ is a pivot, then $H$ is spherical if and only if $\chi(\omega) \in \{\pm 1\}$ for all $\chi\in \Alg(H, \ku)$.
Assume that $H$ is \fd{}; then $H$ is spherical if and only if $\omega$ is involutive. For, $\omega -\omega^{-1}\in \bigcap_{\chi\in \Alg(H, \ku)}\Ker \chi = \Rad H$, and $\Rad H \cap \ku[\omega] = 0$.
\end{Ex}

\subsection{Spherical Categories}

A monoidal rigid category $\C$ is pivotal when $X^{**}$ is  monoi\-dally isomorphic to $X$ \cite{FY-adv};
this implies that the left and right dualities coincide.
For instance, if $H$ is a Hopf algebra and $\omega\in G(H)$ is a pivot, then $\Rep H$ is pivotal \cite[Proposition 3.6]{BaW-adv}.
In a pivotal category $\C$, there are left and right traces $\tr_L, \tr_R: \End (X) \to \End (\uno)$, for any $X\in \C$.
If $\C = \Rep H$ and $V\in \Rep H$, then these traces are defined by
\begin{align}\label{eq:tr-sph-Ha}
\tr_L(\vartheta) &= \tr_V(\vartheta \omega), & \tr_R(\vartheta)  &= \tr_V(\vartheta \omega^{-1}), & \vartheta &\in \End_H(V).
\end{align}
A spherical category is a pivotal one where the left and right traces coincide. Thus, $\Rep H$ is a spherical category, whenever $H$ is a spherical Hopf algebra.

\begin{Rem}\label{obs:spherical-subcat}
If $\D$ is a rigid monoidal (full) subcategory of a spherical category $\C$, then $\D$ is also spherical.
\end{Rem}

\subsection{Quantum dimensions}\label{subsec:qdim}

The \emph{quantum dimension} of an object $X$ in a spherical category $\C$ is given by $\qdim V := \tr_L(\id_X)$. In particular, if $H$ is a spherical Hopf algebra, then
\begin{align}\label{eq:qdim-sph-Ha}
\qdim M &= \tr_M(\omega) = \tr_M(\omega^{-1}), & M&\in \Rep H.
\end{align}

\subsubsection{} If $\C$ is a spherical tensor category, then $\End (\uno)=\ku$,
and the function $V\mapsto \qdim V$ is a character of the Grothendieck ring of $\C$.
In fact, this map is additive on exact sequences,  as in  the proof of Proposition \ref{prop:reduccion a irreducibles};
also
\begin{equation}\label{eq:qdim-producto tensorial}
\qdim V\ot W=\qdim V \qdim W, \qquad V,W \in \C.
\end{equation}
In consequence, the quantum dimension of any object in a \emph{finite} spherical tensor category $\C$ is an algebraic integer in $\ku$, see \cite[Corollary 1.38.6]{EGNO}.

\subsubsection{}
Let $H$ be a Hopf algebra. Given $L\in \Irr H$, $M\in \Rep H$, we set $(M:L) =$ multiplicity of $L$ in $M$ (i.~e. the number of times that $L$ appears as a Jordan-H\"older factor of $M$). Assume that $(H, \omega)$ is spherical. Let $M\in \Rep H$. Then
\begin{align}\label{eq:qdim-formula}
\qdim M &= \sum_{L\in \Irr H} (M:L)\qdim L.
\end{align}
Here is a way to compute the quantum dimension of $M$: consider the decomposition $M = \oplus_{\rho \in \Irr G(H)} M_\rho$ into isotypical components of the restriction of $M$ to $G(H)$. Since $\omega \in Z(G(H))$, it acts by a scalar $z_\rho$ on the $G(H)$-module affording $\rho \in \Irr G(H)$. Hence
\begin{align}\label{eq:qdim-formula-isotypical}
\qdim M &= \sum_{\rho \in \Irr G(H)} z_\rho \dim M_\rho.
\end{align}
See \cite{CW08, CW10} for a Verlinde formula and other  information on the computation of the quantum dimension in terms of the Grothendieck ring.

\subsubsection{}
If $H$ is a non-semisimple spherical Hopf algebra, then  $\qdim M = 0$ for
any $M \in \Rep H$ projective \cite[Proposition 6.10]{BaW-tams}. More generally, the following result appears in the proofs of \cite[2.16]{EO}, \cite[1.53.1]{EGNO}.

\begin{prop}\label{prop:qdim-proyectivos} Let $\C$ be a non-semisimple pivotal tensor category.
Then $\qdim P=0$ for any projective object $P$. \qed
\end{prop}

\subsection{The non-degenerate quotient}\label{subsec:nondeg-quotient}

Let $\C$ be an additive $\ku$-linear spherical category with $\End (\uno)\simeq \ku$.
For any two objects $X, Y$ in $\C$ there is a bilinear pairing
\begin{align*}
\Theta: \Hom_{\C}(X, Y) \times \Hom_{\C}(Y, X) &\to \ku, & \Theta (fg) &= \tr_L(fg) = \tr_R(gf);
\end{align*}
$\C$ is non-degenerate if $\Theta$ is, for any $X,Y$. By \cite[Theorem 2.9]{BaW-adv}, see also \cite{T-ijm},
any additive spherical category $\C$ gives rise a factor category $\Co$, with
the same objects\footnote{This is a bit misleading, as non-isomorphic objects in $\C$ may became isomorphic in $\Co$.} as $\C$ and morphisms $\Hom_{\Co}(X, Y) := \Hom_{\C}(X, Y)/\cJ(X,Y)$, $X,Y \in \C$, where
\begin{align}
\cJ(X,Y) &= \{f \in  \Hom_{\C}(X, Y): \tr_L(fg) = 0, \forall g\in  \Hom_{\C}(Y, X) \}.
\end{align}
The category $\Co$ is an additive non-degenerate spherical category, but it is not necessarily abelian, even if $\C$ is abelian.
Clearly, the quantum dimensions in $\Rep H$ and $\Repo H$ are the same.
See \cite{Kn} for a general formalism of tensor ideals, that encompasses the construction above.

\smallbreak
We now give more information on $\Repo H$ following \cite[Proposition 3.8]{BaW-adv}.
Let us first point out some precisions on the terminology used
in the literature on additive categories, see e.~g. \cite{harada} and references therein. We shall stick to this terminology
in what follows. Let $\C$ be an additive $k$-linear category, where $k$ is an arbitrary field.

\begin{itemize}
	\item $\C$ is \emph{semisimple} if the algebras $\End (X)$ are semisimple for all $X\in \C$.

\smallbreak	
	\item An object $X$ is \emph{minimal}\label{page:minimal} if any monomorphism $Y \to X$ is either 0 or an isomorphism; hence $\End(X)$ is a division algebra over $k$.

\smallbreak	
	\item $\C$ is \emph{completely reducible} if every object is a direct sum of minimal ones.

\end{itemize}

Beware that in more recent literature in topology (where it is assumed that $k = \ku$), an object $S$ is said to be \emph{simple} when $\End (S) \simeq \ku$; and `$\C$ is
\emph{semisimple}' means that every object in $\C$ is a direct sum of simple ones. For instance, if $\C = \lmf{A}$, $A$ an algebra,
then a minimal object in $\C$ is just a a simple $A$-module;
then $\End(S)\cong \ku$ by the Schur Lemma. But it is well-known that the converse is not true.

\begin{Ex}
Let $H = \ku\langle
x,g\vert x^2, g^2 -1, gx + xg\rangle$ be the Sweedler Hopf algebra. $H$ has two simple
modules, both one dimensional, namely the trivial $V^+$ and $V^-$, where $g$ acts as $-1$.
Consider the non-trivial extension $V^{\pm}\in\Ext_H^1(V^-,V^+)$, that is
$V^{\pm}=\ku v \oplus \ku w$ with action
\begin{align*}
&g\cdot v=v, &&g\cdot w=-w, &&  x\cdot v=w, &&  x\cdot w =0.
\end{align*}
Then $V^{\pm}$ is not simple and $\End_H(V^{\pm})\cong \ku$. It is easy to see that the indecomposable $H$-modules are $V^+$, $V^-$,  $V^{\pm}$
and $V^{\mp} := (V^{\pm})^*$; hence $\Repo H \simeq \Rep \Z/2$ by Step \ref{le:caracterizacion de los simples}
of the proof of Theorem \ref{prop:repoH-ss} below.
Also, notice that in $\Rep H$ it is not true that all endomorphism algebras are semisimple (just take the regular representation); and, related to this, that
$\Hom_{H} (V^{\mp}, V^{\pm})  \neq 0$.
\end{Ex}

However, the converse above is true when the category is semisimple.

\begin{Rem}\label{obs:harada} \cite[1.1, 1.3]{harada}
Let $\C$ be a semisimple additive $k$-linear category, where $k$ is an arbitrary field.

\begin{enumerate}\renewcommand{\theenumi}{\alph{enumi}}   \renewcommand{\labelenumi}{(\theenumi)}
\item If $\alpha: V \to W$ is not zero, then there exists $\beta,\gamma: W\to V$ such that
$\beta\alpha \neq 0$, $\alpha\gamma \neq 0$. If $\Hom_{\C}(V,W) =0$, then $\Hom_{\C}(W, V) =0$.

\item If $V\in\C$ and $\End(V)$ is a division ring, then $V$ is minimal.

\item If $V, W\in\C$ are minimal and non-isomorphic, then $\Hom_{\C}(V, W) =0$.
\end{enumerate}

\end{Rem}

\pf (a) Assume that $\Hom_{\C}(W, V)\alpha =0$. Set $U = V\oplus W$; then $\End_{\C}U \alpha$ is a nilpotent left ideal of $\End_{\C}U$, hence it is 0.

(b) Let $W\neq0$ and $f\in\Hom(W,V)$ be a monomorphism.
Since $f\neq0$, $\Hom(V,W)\neq0$ by (a). Since $\End(V)$ is a division ring, the map $f\circ-: \Hom(V,W) \to \End (V)$ is surjective. Therefore there exists $g\in\Hom(V,W)$ such that $f\circ g=\id_V$.
On the other hand, the map $\End(W)\rightarrow\Hom(W,V)$, $\tilde{h}\mapsto f\circ\tilde{h}$ is injective. Hence $g\circ f=\id_W$ since $f\circ(g\circ f)=f\circ\id_W$. Thus $W\simeq V$.
(c) follows from (a) at once.
\epf

If $H$ is a spherical Hopf algebra, then we denote by
$\Indesc_q H$ the set of isomorphism classes of indecomposable \fd{} $H$-modules with non-zero quantum dimension.

\begin{Thm}\label{prop:repoH-ss} \cite[3.8]{BaW-adv}  Let $H$ be a spherical Hopf algebra with pivot $\omega$.
Then the non-degenerate quotient $\Repo H$ is a completely reducible spherical tensor category, and
$\Irr\Repo H$ is in bijective correspondence with $\Indesc_q H$.
\end{Thm}

By Remark \ref{obs:harada}, `completely reducible' becomes what is called semisimple in the recent literature.

\begin{proof}
The crucial step is to show that $\Repo H$ is semisimple.

\begin{step}\label{step:baw-ss} \cite[3.7]{BaW-adv}
The algebra $\End_{\Repo H} (X)$ is semisimple for any  $X$ in $\Repo H$.
\end{step}

Indeed, the Jacobson radical $J$ of $\End_{H}(X)$ is contained in $\J(X,X)$. For, if $\vartheta \in J$, then $\vartheta\omega$ is nilpotent, hence $\tr_L(\vartheta) = \tr_V(\vartheta\omega) = 0$.

\smallbreak
As a consequence of Step \ref{step:baw-ss} and Remark \ref{obs:harada},  $X\in \Repo H$ is minimal if and only if $\End_{\Repo H}(X) \simeq \ku$, that is if and only if it is simple. Also, $\Hom_{\Repo H} (S, T)  = 0$, for $S,T\in \Repo H$ simple non-isomorphic.

\begin{step}\label{le:caracterizacion de los simples} $V$ is a simple object in $\Repo H$ iff there exists $W\in\Rep H$ indecomposable with $\qdim W\neq0$ which is isomorphic to $V$ in $\Repo H$.
\end{step}

Assume that $W\in\Rep H$ is indecomposable. If $f\in\End_H(W)$, then $f$ is either bijective or nilpotent by the Fitting Lemma. If also $\qdim W\neq0$, then $\End_{\Repo H}(W)$ is a finite dimensional division algebra over $\ku$, necessarily isomorphic to $\ku$.
Now, assume that $V$ is a simple object in $\Repo H$. Let $\pi\in\End_H(V)$ be a lifting of $\id_V=1\in\End_{\Repo H}(V)\simeq\ku$.
We can choose $\pi$ to be a primitive idempotent. Then the image $W$ of $\pi$ is indecomposable and $\pi_{|W}$ induces an
isomorphism between $W$ and $V$ in $\Repo H$. Again by the Fitting Lemma, $\End_H(W)\simeq\ku\pi_{|W}\oplus \Rad \End_H(W)$.
Hence $\qdim W\neq0$ since $\pi$ is a lifting of $\id_V\in\End_{\Repo H}(V)$.

\begin{step}\label{step:harada-ss-complred} Let $V, W\in\Rep H$ indecomposable with $\qdim V\neq0$, $\qdim W\neq0$,  which are isomorphic in $\Repo H$.
Then $V\simeq W$ in $\Rep H$.
\end{step}

Let $f\in \Hom_H(V, W)$ and $g\in \Hom_H(W, V)$ such that $gf = \id_V$ in $\Repo H$; that is
\begin{align}\label{eq:step3-aux}
\tr_V (\vartheta(gf - \id) \omega) &= 0 & \text{for every }\vartheta\in \End_H(V).
\end{align}
Since $V$ is indecomposable, $gf$ is invertible in $\End(V)$, or otherwise \eqref{eq:step3-aux} would fail for $\vartheta = \id$.
Thus $g$ is surjective and $f$ is injective. But $W$ is also indecomposable, hence $f$ is surjective and $g$ is injective, and both are isomorphisms.

\medbreak
To finish the proof of the statement about the irreducibles, observe that an indecomposable $U\in \Rep H$ with $\qdim U =0$ satisfies $U\simeq 0$ in $\Repo H$. Since any $M\in \Rep H$
is a direct sum of indecomposables, we see that $M$ is isomorphic in $\Repo H$ to a direct sum of indecomposables with non-zero quantum dimension.

\medbreak
Finally observe that the additive $\ku$-linear category $\Repo H$, being isomorphic to a direct sum of copies of $\Vect \ku$, is abelian.
\end{proof}

Here is a consequence of the Theorem: let $\iota: V \hookrightarrow W$ be a proper inclusion of indecomposable $H$-modules with $\qdim V\neq0$, $\qdim W\neq0$. Then
$\iota \in \J(V,W)$.

\begin{Rem}\label{rem:GK} From Theorem \ref{prop:repoH-ss} we see  the relation between the constructions of \cite{BaW-adv} and \cite{GK-inv}. For, let $\C = \Rep H$ and let $\C^0$, resp. $\C^\perp$, be the full subcategory whose objects are direct sums of indecomposables with quantum dimension $\neq 0$, resp. 0. Then $\Repo H$ is the quotient of $\C$ by $\C^\perp$ as described in \cite[Section 1]{GK-inv}.
\end{Rem}

Even when $H$ is \fd, $\Irr\Repo H$ is not necessarily finite. It is then natural to look at suitable subcategories of $\Rep H$ that give rise to finite tensor subcategories of $\Repo H$. A possibility is tilting modules, that proved to be very fruitful in the case of quantum groups at roots of one.
We shall discuss this matter in Section \ref{sec:titling2}.

\subsection{Pointed or copointed pivotal Hopf algebras}\label{subsec:pointed-gral}

\subsubsection{} Let $H$ be a pointed Hopf algebra and set $G = G(H)$. We assume that $H$ is generated by group-like and skew-primitive elements. For $H$ finite-dimensional, it was conjectured
that this is always the case \cite[Conjecture 1.4]{AS-adv}. So far this is true in all known cases, see
\cite{Ang3,AGI} and references therein. As explained in Subsection \ref{subsec:hopf},
there exist $g_1, \dots, g_{\theta}\in G$
and $\rho_i \in \Irr C_G(g)$, $1 \le i \le \theta$, such that the infinitesimal braiding of $H$ is
$$M(\oc_{g_1}, \rho_1) \oplus \dots \oplus M(\oc_{g_\theta}, \rho_\theta).$$

\begin{lema}\label{lema:pointed-omega-square-antipode} Let $\omega \in G$. Then the following are equivalent:
\begin{enumerate}\renewcommand{\theenumi}{\alph{enumi}}   \renewcommand{\labelenumi}{(\theenumi)}
\item $\omega$ is a pivot.

\item $\omega\in Z(G)$ and $\rho_i(\omega) = \rho_i(g_i)^{-1}$,  $1 \le i \le \theta$.
\end{enumerate}
\end{lema}

\pf
There exist $x_1, \dots, x_{\theta} \in H$ such that $\Delta(x_i) = x_i \ot 1 + g_i\ot x_i$, $1 \le i \le \theta$, and $H$ is generated by $x_1, \dots, x_{\theta}$ and $G$ as an algebra. Now $\Ss^2(x_i) = g_i^{-1}x_ig_i = \rho_i(g_i^{-1}) x_i$.
\epf

\subsubsection{}
Let $G$ be a finite group and $\delta_g$ be the characteristic function of the subset $\{g\}$ of $G$. If $M\in{}_{\ku^G}\M$, then $M=\oplus_{g\in\Supp M}M[g]$ where
$$
M[g]:=\delta_g\cdot M\,\mbox{ and }\,\Supp M:=\{g\in G:M[g]\neq0\}.
$$

\begin{lema}\label{lema:copointed-omega-square-antipode} Let $H$ be a \fd{} copointed Hopf algebra over $\ku^G$ and $\omega=\sum_{g\in G}\omega(g)\delta_g \in G(\ku^G)$. The following are equivalent:
\begin{enumerate}\renewcommand{\theenumi}{\alph{enumi}}   \renewcommand{\labelenumi}{(\theenumi)}
\item $\omega$ is a pivot.

\item $\cS^2(x) = \omega(g)x$ for all $x\in H[g]$, $g\in G$.
\end{enumerate}
\end{lema}

\begin{proof} Consider $H$ as a $\ku^G$-module via the adjoint action; then $\delta_tx=x\delta_{g^{-1}t}$ for  $x\in H[g]$, $g, t\in G$ \cite[3.1 (b)]{AV}. Hence  $\omega x\omega^{-1}=\omega(g)x$.
\end{proof}

\subsection{Spherical Hopf algebras with involutory pivot}\label{sec:inv-pivot}

There are many examples of Hopf algebras with involutory pivot.

\subsubsection{}
Let $H$ be a \fd{} pointed Hopf algebra with $G(H)$ abelian. Then its infinitesimal braiding $V$ is a braided vector space
of diagonal type with matrix $(q_{ij})_{1\le i,j \le \theta}$, $\theta = \dim V$. Assume that $q_{ii} = -1$, $1\le i \le \theta$. The list of all braided vector spaces with
this property and such that the associated Nichols algebra is \fd{} can be easily extracted from the main result of \cite{He-adv}.
We apply  Lemma \ref{lema:pointed-omega-square-antipode} because $H$ is generated by group-like and skew-primitive elements \cite{Ang3}.  Hence,
if $V$ belongs to this list, then $H$ is a spherical Hopf algebra with involutory pivot, eventually adjoining a group-like if necessary.
The argument also works when $G(H)$ is not abelian but the infinitesimal braiding is a direct sum of one-dimensional Yetter-Drinfeld modules (one often says that the infinitesimal braiding \emph{comes from the abelian case}).

\subsubsection{} Let $H$ be a \fd{} pointed Hopf algebra with $G(H)$ not abelian and such that the infinitesimal braiding does not come from the abelian case.
In all the examples of such infinitesimal braidings that are known\footnote{The list of all known examples is in \texttt{http://mate.dm.uba.ar/\textasciitilde matiasg/zoo.html}, see also \cite[Table 1]{GHV}, except for one example discovered later \cite[Proposition 36]{HLV}.},
we may apply  Lemma \ref{lema:pointed-omega-square-antipode} because $H$ is generated by group-like and skew-primitive elements.
Also, in all examples except one in \cite{HLV}, the scalar in Lemma \ref{lema:pointed-omega-square-antipode} (b) is -1. Thus $H$ is a spherical Hopf algebra with involutory pivot, eventually adjoining a group-like if necessary.

\subsubsection{} Let $H$ be a \fd{} copointed Hopf algebra over $\ku^G$, with $G$ not abelian. Lemma \ref{lema:copointed-omega-square-antipode}
makes it easy to check whether $H$ has an involutive pivot. For instance, the Hopf algebras $\cA_{[\ba]}$, for $\ba\in\gA_3$,
introduced in \cite{AV}, have an involutory pivot. These Hopf algebras are liftings of
$\toba(V_3)\#\ku^{\Sn_3}$, so they have dimension 72; but they are not quasi-triangular
\cite{AV}.

Also, $\Irr\Repo\cA_{[0]}$ is infinite. Namely, $\cA_{[0]} \simeq
\toba(V_3)\#\ku^{\Sn_3}$ and for $\lambda\in\ku$ we define the $\cA_{[0]}$-module
$M_\lambda=\langle m_{g}:e\neq g\in\Sn_3\rangle$ by
\begin{align*}
\delta_h\cdot m_g=\delta_{h}(g)m_g\quad\mbox{and}\quad\xij{ij}\cdot m_g=\begin{cases}
                                                     0 &\mbox{if }\sgn g=-1,\\
                                                     \lambda_{(ij),g}\,m_{(ij)g}&\mbox{if }\sgn g=1,
                                                     \end{cases}
\end{align*}
where $\lambda_{(ij),g}=1$ except to $\lambda_{(12),(123)}=\lambda$. Then $M\in\Indesc_q\cA_{[0]}$ with $\qdim M_{\lambda}=-1$ for all $\lambda\in\ku$ and $M_\lambda\simeq M_\mu$ iff $\lambda=\mu$. Analogously, we can define $N_\lambda\in\Indesc_q\cA_{[0]}$ with basis $\langle n_{g}:(12)\neq g\in\Sn_3\rangle$, $\qdim N_{\lambda}=1$ and which are mutually not isomorphic.

\begin{Rem} The dual of $\cA_{[0]}$ is $\toba(V_3)\#\ku\Sn_3$, which is not pivotal because $\Sn_3$ is centerless. Compare with the main result of \cite{p}, where it is shown that the dual of a semisimple spherical Hopf algebra is again spherical.
In fact, the dual of $\cA_{[\ba]}$ is not pivotal for any $\ba\in\gA_3$. If $\ba$ is generic, then $(\cA_{[\ba]})^*$ has no non-trivial group-likes \cite[Theorem 1]{AV2}.
If $\ba$ is sub-generic, then the unique non-trivial group-like $\zeta_{(12)}$ of $(\cA_{[\ba]})^*$ is not a pivot, see \cite[Lemma 8]{AV2} for notations.
Namely, if $g\neq(12),e$, then
\begin{align*}
\zeta_{(12)}\rightharpoonup\delta_g\leftharpoonup\zeta_{(12)}=\sum_{t,s\in\Sn_3}\delta_{s}((12))\,\delta_{s^{-1}t}\,\delta_{t^{-1}g}((12))=\delta_{(12)g(12)}\neq\cS^2(\delta_g).
\end{align*}
\end{Rem}

\subsubsection{}
Let $H$ be a spherical Hopf algebra with involutory pivot $\omega$. Then

\begin{itemize}
  \item The quantum dimensions are integers.
  \item If $\chi$ is a representation of dimension one, then $\qdim \ku_{\chi} = \chi(\omega)$.
  \item If $H$ is not semisimple, then at least one module has negative quantum dimension.
\item Assume that there exists $L\in \Irr H$ such that $\qdim L' > 0$ for all $L'\in \Irr H$, $L'\neq L$. Then $\qdim L < 0$.
\end{itemize}

\begin{prop}\label{prop:spherical-involutory}
Let $\C$ be a fusion subcategory of $\Repo H$, where $H$  is a spherical Hopf algebra with involutory pivot. Then there exists a semisimple quasi-Hopf algebra $K$
such that $\C \simeq \Rep K$ as fusion categories.
\end{prop}

\pf The quantum dimensions are integers, because the pivot is involutory, and positive by \cite[Corollary 2.10]{ENO}; here we use that $\C$ is spherical, see Remark \ref{obs:spherical-subcat}. Then \cite[Theorem 8.33]{ENO} applies. Indeed, the Perron-Frobenius and quantum dimensions here coincide, see e.~g. the proof of \cite[Proposition 8.23]{ENO}.
\epf

We are inclined to believe, because of some computations in  examples, that the quasi-Hopf algebra $K$ in the statement is actually a Hopf algebra quotient of $H$.

\subsection{Ribbon Hopf algebras}\label{subsec:ribbon}
This is a distinguished class of spherical Hopf algebras.
Let  $(H, \Rc)$ be a quasi-triangular Hopf algebra \cite{Dr1}.
We denote the universal matrix as $\Rc = \Rc\^{1} \ot \Rc\^{2}$.
The \emph{Drinfeld element} is
\begin{equation}\label{eqn:drinfeldelement}
\ub = \Ss(\Rc\^{2})\Rc\^{1}.
\end{equation}
Let $Q = \Rc_{21}\Rc$.  The Drinfeld element is invertible and satisfies
\begin{align}\label{eq:drinfeld-element1}
\Delta(\ub) = Q^{-1}(\ub\ot \ub) &= (\ub\ot \ub) Q^{-1},  & \ub^{-1} &=  \Rc\^{2}\Ss^2(\Rc\^{1}),
\\\label{eq:drinfeld-element2}
g :=\ub\Ss(\ub)^{-1} &= \Ss(\ub)^{-1}\ub\in G(H), & \ub\Ss(\ub) &\in Z(H),
\\\label{eq:drinfeld-element3}
\Ss^2(h) &= \ub h\ub^{-1},  & \Ss^4(h) &= g h g^{-1},
\end{align}
for any $h\in H$.

\begin{definition}\label{defi:ribbon}\cite{RT-cmp}
A quasi-triangular Hopf algebra $(H, \Rc)$ is \emph{ribbon} if there exists $\vb\in Z(H)$, called
the \emph{ribbon element}, such that
\begin{align}\label{eq:ribbon-element1}
\vb^2 &= \ub\Ss(\ub),  & \Ss(\vb) &= \vb, & \Delta(\vb) &= Q^{-1}(\vb\ot \vb).
\end{align}
The ribbon element is not unique but it is determined up to multiplication by an element in $\{g\in G(H) \cap Z(H): g^2 = 1\}$.
\end{definition}

Let $H$ be a ribbon Hopf algebra. It follows easily that $\omega = \ub\vb^{-1}\in G(H)$  and $\Ss^2(h) = \omega h \omega^{-1}$  for all $h\in H$; that is, $\omega$ is a pivot. Actually, $H$ is spherical \cite[Example 3.2]{BaW-adv}. In fact, the concept of {\it quantum trace}
is defined in any ribbon category  using the braiding, see for example \cite[XIV.4.1]{ka}. By \cite[XIV.6.4]{ka}, the {\it quantum trace} of $\Rep H$ coincides with $\tr_L$, cf. \eqref{eq:tr-sph-Ha}. Moreover, \cite[XIV.4.2 (c)]{ka} asserts that $\tr_L = \tr_R$.

\bigbreak
There are quasi-triangular Hopf algebras that are not ribbon, but the failure is not difficult to remedy by adjoining a group-like element \cite[Theorem 3.4]{RT-cmp}. Namely, given a quasi-triangular Hopf algebra $(H, \Rc)$, let $\wH = H \oplus H\vb$, where $\vb$ is a formal element not in $H$. Then $\wH$ is a Hopf algebra with product, coproduct, antipode and counit defined for $x,x',y,y' \in H$ by
\begin{align}\label{eq:ribbon-ext-producto}
(x + y\vb) \cdot (x' + y' \vb) &= (xx' + yy' \ub\Ss(\ub)) + (xy' + yx') \vb,
\\\label{eq:ribbon-ext-coproducto}
\Delta(x + y\vb) &= \Delta(x) + \Delta(y)Q^{-1}(\vb\ot \vb),
\\\label{eq:ribbon-ext-antipoda}
\Ss(x + y\vb) &= \Ss(x) + \Ss(y)\vb, \qquad  \cou(x + y\vb) = \cou(x) + \cou(y).
\end{align}
Clearly, $H$ becomes a Hopf subalgebra of $\wH$; it can be shown then that $\Rc$ is a universal $R$-matrix for $\wH$ and that $\vb$ is a ribbon element for
$(\wH,\Rc)$. See \cite[Theorem 3.4]{RT-cmp}. We shall say that $\wH$ is the \emph{ribbon extension} of $(H, \Rc)$.

\begin{Rem}\label{rem:ribbon-ext} (Y. Sommerh\"auser, private communication). The ribbon extension fits into an exact sequence  of Hopf algebras
$H \hookrightarrow\wH \twoheadrightarrow \ku [\Z/2]$, which is cleft. Namely, let $\xi$ be the generator of $\Z/2$ and define
\begin{align*}
\xi \rightharpoonup x &= \Ss^2(x), & x&\in H,& \sigma(\xi^i \ot \xi^j)&= g^{ij},& i,j&\in \{0,1\}.
\end{align*}
Then the crossed product defined by this action and cocycle together with the tensor product of coalgebras is a Hopf algebra $H \#_{\sigma}\ku[\Z/2]$, see for instance \cite{AD}. Now the map $\psi: H \#_{\sigma}\ku[\Z/2] \to \wH$, $\psi(x\# \xi^i) = x (\Ss(u^{-1})\vb)^i$, is an isomorphism of Hopf algebras.
\end{Rem}

In conclusion, any \fd{} Hopf algebra $H$ gives rise to a ribbon Hopf algebra, namely the ribbon extension of its Drinfeld double:
$$
\xymatrix{H \ar@{~>}[0,1]  & D(H) \ar@{~>}[0,1] & \widetilde{D(H)}.}
$$

\begin{Rem}\label{rem:ribbon-double} A natural question is whether the Drinfeld double itself is ribbon; this was addressed in \cite{KR}, where the following results were obtained. Let $H$ be a \fd{} Hopf algebra, $g\in G(H)$ and $\alpha\in G(H^*)$ be the distinguished group-likes\footnote{These control the passage from left to right integrals.}. The celebrated Radford's formula for the fourth power of the antipode \cite{radford-ajm} states that
\begin{align}\label{eq:radford-S4}
\Ss^4(h) &= g (\alpha \rightharpoonup h\leftharpoonup \alpha^{-1})g^{-1},  & h &\in H.
\end{align}
Here $\rightharpoonup $, $\leftharpoonup$ are the transposes of the regular actions.

\begin{enumerate}\renewcommand{\theenumi}{\alph{enumi}}   \renewcommand{\labelenumi}{(\theenumi)}
\item \cite[Theorem 2]{KR} Suppose that $(H, \Rc)$ is quasi-triangular and that $G(H)$ has odd order. Then $(H, \Rc)$ admits a (necessarily unique)
ribbon element if and only if $\Ss^2$ has odd order.

\bigbreak
\item \cite[Theorem 3]{KR}  $(D(H), \Rc)$ admits a ribbon element if and only if there exist $\ell\in G(H)$ and $\beta\in G(H^*)$ such that
\begin{align}\label{eq:kauffman-radford-S2}
\ell^2 &= g, & \beta^2 &= \alpha, & \Ss^2(h) &= \ell (\beta \rightharpoonup h\leftharpoonup \beta^{-1})\ell^{-1},  & h &\in H.
\end{align}
\end{enumerate}
\end{Rem}

\bigbreak
\subsection{Cospherical Hopf algebras}

It is natural to look at the notions that insure that the category of comodules of a Hopf algebra is pivotal or spherical. This was done in \cite{Bi, o}.

\begin{Def}\label{def:cospherical-ha}  A \emph{cospherical} Hopf algebra is a pair $(H, t)$, where $H$ is a Hopf algebra and  $t\in \Alg(H, \ku)$ is such that
\begin{align}\label{eq:copivot-square-antipode}
\Ss^2(x\_1)t(x\_2)&= t(x\_1)x\_2,  & x&\in H, \\
\label{eq:copivot-traza}
\tr_V((\id_V\otimes t)\rho_V \vartheta) &= tr_V((\id_V\otimes t^{-1})\rho_V \vartheta),  & \vartheta &\in \End_H(V), &&
\end{align}
for all $V\in \Corep H$.
We  say that $t\in \Alg(H, \ku)$ is a \emph{copivot} when it satisfies \eqref{eq:copivot-square-antipode}; pairs $(H, t)$ with $t$ a
copivot are called \emph{copivotal Hopf algebras}.
A \emph{cospherical element} is a copivot that fulfills \eqref{eq:copivot-traza}. Let $H$ be a cospherical Hopf algebra. Then the category $\Corep H$ is spherical; in fact, the left and right traces are given by the sides of \eqref{eq:copivot-traza}.
\end{Def}

The set $\Alg(H, \ku)$ is a subgroup of the group $\Hom_{\star}(H, \ku)$ of convolution-invertible linear functionals,
which in turn acts on $\End (H)$  on both sides.
Hence \eqref{eq:copivot-square-antipode} can be written as $\Ss^2 * t=t* id_H$ or else as $\Ss^2=t*id_H*t^{-1}$.
The copivot is not unique but it is determined up to multiplication by an element in $\Alg(H, \ku)$ that centralizes $\id_H$.
The antipode of a copivotal Hopf algebra is bijective, with inverse given by $\Ss^{-1}(x)=\sum t^{-1}(x\_1)\Ss(x\_2)t(x\_3)$, $x\in H$.

\smallbreak
The following statement is proved exactly as Proposition \ref{prop:reduccion a irreducibles}.

\begin{prop}\label{prop:reduccion a coirreducibles}
A copivotal Hopf algebra $(H,t)$ is cospherical if and only if
\eqref{eq:copivot-traza} holds for all simple $H$-comodules. \qed
\end{prop}

\begin{Exs}
\begin{enumerate}\renewcommand{\theenumi}{\alph{enumi}}   \renewcommand{\labelenumi}{(\theenumi)}
\item Assume that $H$ is \fd. Then $H$ is copivotal (resp., cospherical) iff $H^*$ is pivotal (resp., spherical).

\smallbreak\item  Any involutory Hopf algebra is cospherical with $t = \cou$.

\smallbreak
\item  A copivotal Hopf algebra with involutive copivot is cospherical.

\smallbreak
\item  Condition \eqref{eq:copivot-square-antipode} is multiplicative on $x$; also, it holds for $x\in G(H)$.

\smallbreak
\item\label{item:pointed-pivotal} Let $H$ be a pointed Hopf algebra generated as an algebra by $G(H)$ and a family $(x_i)_{i\in I}$, where $x_i$ is $(g_i,1)$
skew-primitive.
Assume that $g_ix_ig_i^{-1} = q_ix_i$, with $q_i\in \ku^{\times} \setminus \{1\}$ for all $i\in I$.
If $t\in \Alg(H, \ku)$, then $t(x_i) = 0$, $i\in I$. Hence $t$
is a  copivot iff $t(g_i) = q_{i}^{-1}$, for all $i\in I$.

\smallbreak
\item Let $H$ be a pointed Hopf algebra as in item \eqref{item:pointed-pivotal} and $t\in \Alg(H, \ku)$ a copivot. Then $H$ is cospherical iff
$t(g) \in \{\pm 1\}$ for all $g\in G(H)$.

\smallbreak
\item
The notion of coribbon Hopf algebra is formally dual to the notion of ribbon Hopf algebra, see \cite{h, lt}. Coribbon Hopf algebras
 are  cospherical. For instance, the quantized function algebra $\Oc_q(G)$ of a semisimple algebraic group is cosemisimple and coribbon,
 when $q$ is not a root of 1.

\end{enumerate}

\end{Exs}

We recall now the contruction of universal copivotal Hopf algebras.

\begin{Def} \cite{Bi}
Let $F\in GL_n(k)$. The  Hopf algebra $H(F)$ is the universal algebra with generators $(u_{ij})_{1\leq i,j\leq n}$, $(v_{ij})_{1\leq i,j\leq n}$  and relations
\begin{align*}
uv^t&=v^tu=1, & vFu^tF^{-1}&= Fu^tF^{-1}v=1.
\end{align*}

The comultiplication is determined by $\Delta (u_{ij})= \sum_k u_{ik}\otimes u_{kj}$, $\Delta (v_{ij})= \sum_k v_{ik}\otimes v_{kj}$
and the antipode by $\Ss(u)=v^t$, $\Ss(v)=Fu^tF^{-1}$. The Hopf algebra $H(F)$ is copivotal, the copivot being $t_F(u)=(F^{-1})^t$ and $t_F(v)=F$.

\smallbreak
Let $H$ be a Hopf algebra provided with $V\in \Corep H$ of dimension $n$ such that $V\cong V^{**}$.
Then there exist a matrix $F\in GL_n(k)$, a coaction $\beta_V:V\to V\otimes H(F)$ and a Hopf algebra morphism $\pi:H(F)\to H$ such that $(\id_V\otimes \pi)\beta_V=\rho_V$.
\end{Def}

In conclusion, we would like to explore semisimple tensor categories arising from cospherical, not cosemisimple, Hopf algebras. The first step is the
dual version of Theorem \ref{prop:repoH-ss}, which is proved exactly in the same way.

\begin{Thm}\label{prop:corepoH-ss}  Let $H$ be a cospherical Hopf algebra.
Then the non-degene\-rate quotient $\Corepo H$ of $\Corep H$ is a completely reducible spherical tensor category, and
$\Irr\Corepo H$ is in bijective correspondence with the set of isomorphism classes of indecomposable
\fd{} $H$-comodules with non-zero quantum dimension. \qed
\end{Thm}

\section{Tilting modules}\label{sec:titling2}

The concept of tilting modules appeared in \cite{BB1} and was extended to quasi-hereditary algebras in \cite{Ri}.
Observe that \fd{} quasi-hereditary Hopf algebras are semisimple. Indeed, quasi-hereditary  algebras have finite global dimension, but a \fd{} Hopf algebra is a Frobenius algebra, hence it has global dimension 0 or infinite. Instead, the context where the recipe
of tilting modules works is a suitable category of modules, or comodules, of an infinite dimensional Hopf algebra.
The relevant examples are: algebraic semisimple groups over an algebraically closed field of positive characteristic
(the representations are comodules over the Hopf algebra of rational functions),
quantum groups at roots of one and the category $\Oc$ over a semisimple Lie algebra \cite{A, AP, Do1, GM, Mat}.
The main features  are:

\smallbreak
$\bullet$ The suitable category of representations is not artinian, and the simple modules are parameterized by dominant weights; the set of dominant weights admits a total order that refines the usual partial order.
To fit into the framework of quasi-hereditary algebras,  subcategories of modules with weights in suitable subsets are considered; this allows to define  Weyl modules $\Delta (\lambda)$, dual Weyl modules $\nabla(\lambda)$, and eventually tilting modules $T(\lambda)$, for $\lambda$ a dominant weight. Usually these constructions are performed in an \emph{ad-hoc} manner, not through quasi-hereditary algebras, albeit those corresponding to this situation are studied in the literature under the name of Schur algebras.

\smallbreak
$\bullet$ The tensor product of two tilting modules and the dual of a tilting module are again tilting. The later statement is trivial, the former requires a delicate proof.

\smallbreak
$\bullet$ There is an \emph{alcove} inside the chamber defined by the positive roots and bounded by an affine hyperplane. If $\lambda$ is in the alcove, then the simple module satisfies $L(\lambda) = \Delta(\lambda)$, hence it is the  tilting $T(\lambda)$. The tilting modules $T(\lambda)$ outside the alcove are projective, hence have zero quantum dimension. Thus, the fusion category looked for is spanned by the tilting modules in the alcove.

\smallbreak
$\bullet$ The fusion rules between the tilting modules is given by the celebrated Verlinde formula \cite{Ve} or a modular version,
see \cite{AP, Mat}.

\smallbreak
We would like to adapt these arguments to categories of representations of certain Hopf algebras $H$ arising from \fd{} Nichols of diagonal type.
The Hopf algebra $H$ would be the Drinfeld double, or a variation thereof, of the bosonization of the corresponding Nichols algebra with a suitable abelian group.
We would like to solve the following points:

\begin{itemize}
\item The set of  irreducible objets in $\Rep H$ (or some appropriate variation) should split as a filtered  union $\Irr H = \bigcup_{A\in \cA} A$; each $A$ spans an artinian subcategory where tilting modules $\Tc_A$ can be computed.

\item Define the category $\Tc_H$ of tilting modules over $H$ as the union of the various $\Tc_A$; this should be a semisimple category.

\item the category  $\Tc_H$ of tilting modules is stable by tensor products and duals.

\item It is possible to determine which irreducible tilting modules have non-zero quantum dimension; there are a finite number of them.

\item The fusion rules are expressed through a variation of the Verlinde formula.
\end{itemize}

Provided that these considerations are  correct, the full subcategory  of $\Repo H$ generated by the indecomposable tilting modules with non-zero quantum dimension, is a fusion category. In this way, we hope to obtain new examples of non-integral fusion categories.

\subsection{Quasi-hereditary algebras and tilting modules}\label{sec-app:tilting}

Tilting modules work for our purpose because they span a completely reducible category already in $\Rep H$. We think it is worthwhile to recall the main definitions of the theory of (partial) tilting modules over quasi-hereditary algebras, due to Ringel \cite{Ri}. A full exposition is available in \cite{Do2}.

Let $A$ be an artin algebra. Consider a family $\Theta = (\Theta(1), \dots, \Theta(n))$ of $A$-modules such that
\begin{align}\label{eq:ext-0}
\Ext^1_A(\Theta(j), \Theta(i)) &= 0, & j&\geq i.
\end{align}
We denote by $\F(\Theta)$  the full subcategory of $\lm{A}$ with objects
$M$ that admit a filtration with sub-factors in $\Theta$.
We fix a numbering (that is, a total order) of $\Irr A$: $L(1), \dots, L(n)$. We set
\begin{align*}
P(i) &= \text{ projective cover of } L(i),
\\ Q(i) &= \text{ injective hull of } L(i),
\\ \Delta(i) &= P(i)/U(i), \quad \text{ where } U(i) = \sum_{j>i}\quad \sum_{\alpha\in \Hom(P(j), P(i))} \Imm \alpha,
\\ \nabla(i) &=  \bigcap_{j>i}\quad \bigcap_{\beta\in \Hom(Q(i), Q(j))} \ker \beta,
\end{align*}
$1\le i \le n$. Let $\Delta = (\Delta(1), \dots, \Delta(n))$, $\nabla = (\nabla(1), \dots, \nabla(n))$; these satisfy \eqref{eq:ext-0} and then \cite[Theorem 1]{Ri} applies to them.

\begin{Def}\label{def:quasi-hereditary} The artin algebra $A$ is \emph{quasi-hereditary} provided that
 $_{A}A \in \F(\Delta)$, and $L(i)$ has multiplicity one in  $\Delta(i)$,  $1 \le i \le n$.
\end{Def}

\begin{Rem}\label{obs:quasi-hereditary-alternative} Quasi-hereditary algebras were introduced by Cline, Parshall and Scott, see e. g. \cite{CPS}. There are some alternative definitions.

\begin{enumerate}\renewcommand{\theenumi}{\alph{enumi}}   \renewcommand{\labelenumi}{(\theenumi)}
\item\label{item:quasi-hereditary-alternative inductive} An ideal $J$ of an artin algebra $A$ is \emph{hereditary} provided that
\begin{itemize}
  \item $J\in \lm{A}$ is projective,
  \item $\Hom_A(J, A/J) = 0$ (Ringel assumes $J^2 = J$ instead of this),
  \item $JNJ =0$, where $N$ is the radical of $A$.
\end{itemize}
It can be shown that $A$ is quasi-hereditary iff there exists a chain of ideals $A = J_0 > J_1 > \dots > J_m =0$
with $J_i/J_{i+1}$ hereditary in $A/J_{i+1}$.

\item Also, $A$ is quasi-hereditary iff $\lm{A}$ is a highest weight category, that is the following holds for all $i$:

\begin{itemize}
\item $Q(i)/ \nabla(i)\in \F(\nabla)$.

\item If $(Q(i)/ \nabla(i): \nabla(j)) \neq 0$, then $j>i$.
\end{itemize}
\end{enumerate}

\end{Rem}

For completeness, we include the definitions of tilting, cotilting and basic modules, see e.g. \cite{Ri} and its bibliography. First, a  module $T$ is tilting provided that
\begin{itemize}
  \item it has finite projective dimension;
  \item $\Ext^i(T, T) =0$ for all $i\ge 1$;
  \item for any projective module $P$, there should exist an exact sequence $0 \to P \to T_0 \to \dots \to T_m \to 0$, with all $T_j$ in the additive subcategory generated by $T$, denoted $\add T$.
\end{itemize}

Second, a cotilting module should have
\begin{itemize}
  \item finite injective dimension;
  \item $\Ext^i(T, T) =0$ for all $i\ge 1$;
  \item for any injective module $I$, there should exist an exact sequence $0 \to T_m \to \dots \to T_0 \to I \to 0$, with all $T_j$  in $\add T$.
\end{itemize}

Lastly, a basic module is one with no direct summands of the form $N\oplus N$, with $N\neq 0$.

\smallbreak
Assume that $A$ is a quasi-hereditary algebra and consider the full subcategory $\Tc = \Tc_A = \F(\Delta) \cap \F(\nabla)$ of (partial) tilting modules.
It was shown in \cite[Theorem 5]{Ri} that-- for a quasi-hereditary algebra-- there is a unique basic module $T$, that is tilting and cotilting (in the sense just above), and such that $\Tc$ coincides with $\add T$. The relation between this $T$ and the partial tilting modules is clarified by the following result.

\begin{Thm}\label{thm:tilting-ringel} \cite[Corollary 2, Proposition 5]{Ri} There exist indecomposable (partial) tilting modules $T(i)\in \Tc$,   $1 \le i \le n$, with the following properties:
\begin{itemize}
	\item Any indecomposable tilting module is isomorphic to one of them.
	\item $T = T(1) \oplus \cdots \oplus T(n)$ is the tilting module  mentioned above.
	\item There are exact sequences
\begin{align*}
&0 \longrightarrow \Delta(i) \overset{\beta(i)}\longrightarrow  T(i)\longrightarrow X(i)\longrightarrow 0,
\\
&0 \longrightarrow Y(i) \longrightarrow T(i)\overset{\gamma(i)}\longrightarrow  \nabla(i)\longrightarrow 0,
\end{align*}
where $X(i) \in \F(\{\Delta(j): j < i\})$, $Y(i) \in \F(\{\nabla(j): j < i\})$, $\beta(i)$ is a left $\F(\nabla)$-approximation and $\gamma(i)$ is a right $\F(\Delta)$-approximation, $1\le i \le n$ (see \cite{Ri} for undefined notions). \qed
\end{itemize}
\end{Thm}

\subsection{Induced and produced}\label{subsec:ind-prod}

Let $B \hookrightarrow A$ be an inclusion of algebras. We denote by $\Res_B^A$ the restriction functor from the category $\lm{A}$ to $\lm{B}$.
\subsubsection{}\label{subsub:pro-ind-generalities} The induced and produced modules of $T \in \lm{B}$ are
\begin{align}\label{eq:in-prod}
\Ind_{B}^A T &= A\ot_B T, & \Pro_{B}^A T &= \Hom_B (A, T).
\end{align}
These are equipped with morphisms of $B$-modules $\iota: T \hookrightarrow \Ind_{B}^A T$, given by $\iota (t) = 1\otimes t$ for $t\in T$, and $\pi: \Pro_{B}^A T \twoheadrightarrow  T$, given by $\pi (f) = f(1)$ for $f\in \Hom_B(A,T)$. The following properties are well-known.

\begin{enumerate}\renewcommand{\theenumi}{\alph{enumi}}   \renewcommand{\labelenumi}{(\theenumi)}

\item\label{subsub:ind-univ} $\Hom_B (T, \Res_B^A M) \simeq \Hom_A (\Ind_B^A T, M)$; that is, induction is left adjoint to restriction.

\item\label{subsub:simple-ind-quot}  For every $S\in \Irr A$ there exists $T\in \Irr B$ such that $S$ is a
quotient of $\Ind_{B}^A T$.

\item
\label{subsub:pro-univ} $\Hom_B (\Res_B^A N, T) \simeq \Hom_A (N, \Pro_B^A T)$; that is, production (also called coinduction) is right adjoint to restriction.

\item\label{subsub:simple-pro-subobj} For every $S\in \Irr A$ there exists $T\in \Irr B$ such that $S$ is a
submodule of $\Pro_{B}^A T$.

\end{enumerate}

\subsubsection{}\label{subsub:pro-ind-duality} Assume that $A$ is a finite $B$-module. Then  $\Res_B^A$, $\Ind_B^A$ and $\Pro_B^A$ restrict to functors (denoted by the same name) between the categories $\lmf{A}$ and $\lmf{B}$ of \fd{} modules; \emph{mutatis mutandis}, the preceding points hold in this context.
Assume also that there exists a \emph{contravariant} $\ku$-linear functor $\D: \lmf{A} \to \lmf{A}$ such that $\D (\lmf{B}) \subseteq \lmf{B}$ and admits a quasi-inverse $\Ee: \lmf{A} \to \lmf{A}$, so that $\D$ is an equivalence of categories. It follows at once from the universal properties that
\begin{align}\label{eq:pro-dual-ind}
\Ind_B^A T &\simeq \D (\Pro_B^A \Ee T) \simeq \Ee (\Pro_B^A \D T),
\\\label{eq:ind-dual-pro}
\Pro_B^A T &\simeq \D (\Ind_B^A \Ee T) \simeq \Ee (\Ind_B^A \D T).
\end{align}
Hence $\D (\Ind_B^A T) \simeq \Pro_B^A \D T$, $\D (\Pro_B^A T) \simeq \Ind_B^A \D T$, and so on.

In this setting, consider the following conditions:
\begin{align}
\label{eq:ind-unique}
&\text{For every }S\in \Irr A, \exists \text{ a \emph{unique} } T\in \Irr B \text{ such that }
\Ind_{B}^A T \twoheadrightarrow S.
\\\label{eq:pro-unique}
&\text{For every }S\in \Irr A, \exists \text{ a \emph{unique} } U\in \Irr B \text{ such that } S \hookrightarrow \Pro_{B}^A U.
\\\label{eq:ind-head-simple}
&\text{The head of } \Ind_{B}^A T \text{ is simple for every } T\in \Irr B.
\\\label{eq:pro-socle-simple}
&\text{The socle of } \Pro_{B}^A U \text{ is simple for every } U\in \Irr B.
\end{align}
Then \eqref{eq:ind-unique} $\iff$ \eqref{eq:pro-unique} and \eqref{eq:ind-head-simple} $\iff$ \eqref{eq:pro-socle-simple}. If all these conditions hold, then for any $T\in \Irr B$, there exists a unique $U\in \Irr B$ such that
\begin{align*}
\Ind_B^A T \twoheadrightarrow S \hookrightarrow \Pro_{B}^A U,
\end{align*}
where $S$ is the head of $\Ind_B^A T$ and the socle of $\Pro_{B}^A U$. We set
$U = w_0(T)$.

\subsubsection{}\label{subsub:pro-ind-hopf} Let $K$ be a Hopf subalgebra of a Hopf algebra $H$, with $H$ finite over $K$. Then
\begin{align*}
\Ind_K^H T &\simeq (\Pro_K^H {}^*T)^* \simeq {}^*(\Pro_K^H T^*),
& \Pro_K^H T &\simeq (\Ind_K^H {}^*T)^* \simeq {}^*(\Ind_K^H T^*).
\end{align*}
If $H$ is pivotal, then these formulae are simpler because the left and right duals coincide.

\subsubsection{}\label{subsub:pro-ind-other} Let $\g$ be a simple Lie algebra, $\mathfrak b$ a Borel subalgebra and $q$ a root of unity of odd order, relatively prime to 3 when $\g$ is of type $G_2$. Let  $H = \Ulus_q(\mathfrak g)$ be the Lusztig's $q$-divided power quantized enveloping algebra  and $K = \Ulus_q(\mathfrak b)$. Let $\C$ be the category of \fd{} $H$-modules of type 1, see \cite{APW}, and $\C_{\mathfrak b}$ the analogous category of $K$-modules. Then there are induced and produced functors
$\Ind_K^H, \Pro_K^H: \C_{\mathfrak b} \to \C$.
Then the Weyl and dual Weyl modules are the  produced and induced modules of the simple objects in $\C_{\mathfrak b}$,  parameterized conveniently by highest weights. This allows to define Weyl and dual Weyl filtrations and tilting modules; instead of appealing to Theorem \ref{thm:tilting-ringel}, one establishes the semisimplicity of the category of tilting modules by establishing crucial cohomological results, see \cite{AP} for details.

\subsection{\Fd{} Nichols algebras of diagonal type}\label{subsec:qt- pointed} We continue the analysis started in Subsection \ref{subsec:nichols}.

\subsubsection{} Let $\theta\in \N$ and $\I = \{1, \dots, \theta\}$.
Let $\La$ be a free abelian group with basis $\alpha_1, \dots, \alpha_{\theta}$.
Let $\leq$ be the partial order in $\La$ defined by
\begin{align}\label{eq:partial-order}
\lambda &\leq \mu  \iff \mu - \lambda \in  \La^+ := \sum_{i\in \I} \N_0\alpha_i.
\end{align}

Let us fix a $\Z$-linear injective map $E: \La \to \R$ such that $E(\alpha_i) > 0$ for all $i\in \I$. This induces a total order on $\La$ by $\lambda \preceq \mu \iff E(\lambda) \leq E(\mu)$; clearly, $\lambda \leq \mu$ implies $\lambda \preceq \mu$.

Given a $\La$-graded vector space $M = \oplus_{\lambda \in \La} M_{\lambda}$, the $\lambda$'s such that $M_{\lambda}\neq 0$ are called the weights of $M$; the set of all its weights is denoted $\varPi(M)$.

\subsubsection{} Let $(q_{ij})_{i,j\in \I}$ be a symmetric matrix with entries in
$\ku^{\times}$. Let $(V,c)$ be a braided vector space of diagonal type with matrix $(q_{ij})_{i,j\in \I}$, with respect to a basis $(v_i)_{i\in \I}$. The Nichols algebra $\toba(V)$ has a $\La$-grading determined by $\deg v_i = \alpha_i$, $i\in \I$.
By \cite{Kh}, there exists an ordered set $\bar S$ of homogeneous elements of $T(V)$ and a function $h: \bar S\to \N\cup\{\infty\}$
such that:
\begin{itemize}
	\item The elements of $\bar S$ are hyperletters in $(v_i)_{i\in \I}$.
	\item The projection $T(V) \to \toba(V)$ induces a bijection of $\bar S$ with its image $S$. Denote also by $h: S\to \N\cup\{\infty\}$ the induced function.
	\item The following elements form a basis of $\mathcal{B}(V)$:
\begin{align}\label{eq:pbw}
&\big\{s_1^{e_1}\dots s_t^{e_t}: t \in \N_0, \quad s_1>\dots >s_t,\quad s_i \in S, \quad 0 < e_i<h(s_i)\big\}.
\end{align}
\end{itemize}

When $S$ is finite, two distinct elements in $S$ have different degree, and we can
label the elements in $S$ with a finite subset $\varDelta_+^V$ of $\Lambda_+$;
this is instrumental to define  the root system $\mathcal{R}$ of $(V, c)$ \cite{HS}.

\smallbreak Let $W$ be another braided vector space of diagonal type with matrix $(q_{ij}^{-1})_{i,j\in \I}$, with respect to a
basis $w_1,\dots, w_{\theta}$; we shall consider the $\La$-grading on the Nichols algebra $\toba(W)$ determined by
$\deg w_i = -\alpha_i$, $i\in \I$.

\smallbreak \emph{We assume from now on that  $\dim \toba(V) < \infty$}; hence $\dim \toba(W) < \infty$. Under this assumption,
the connected components of $(q_{ij})$ belong to the list given
in \cite{He-adv}. An easy consequence is that $q_{ii} \neq 1$ and $q_{ij}$ is a root of 1 for all $i,j\in \I$;
this last claim follows because the matrix $(q_{ij})$ is assumed to be symmetric.
Also, $S$ is finite and $h$ takes values in $\N$. Thus \eqref{eq:pbw} says that
\begin{align*}
\varPi(\toba(V)) = \big\{\sum_{s \in S} e_s \deg s, \quad 0 \leq e_s<h(s)\big\}.
\end{align*}
Note that $0 \le \alpha \le \varrho$ for all $\alpha \in \varPi(\toba(V))$, where
\begin{align}\label{eq:deg-top}
\varrho = \sum_{s \in S} (h(s) - 1) \deg s = \deg \toba^{\text{top}}(V) \in \varPi(\toba(V)).
\end{align}

\subsubsection{} A \emph{pre-Nichols algebra} of $V$ is any graded braided Hopf algebra $\mathfrak T$ intermediate between $\toba(V)$ and $T(V)$:  $T(V) \twoheadrightarrow \mathfrak T \twoheadrightarrow \toba(V)$ (Masuoka). The defining relations of the Nichols algebra $\toba(V) = T(V)/\J(V)$ are listed as (40), \dots, (68) in \cite[Theorem 3.1]{Ang3}.  Now we observe
that, since $\dim\toba(V) < \infty$,  the following integers exist:
\begin{align} \label{defn:mij}
-a_{ij}&:= \min \left\{ n \in \mathbb{N}_0: (n+1)_{q_{ii}} (1-q_{ii}^n q_{ij}^2 )=0 \right\}
\end{align}
for all $j \in \I - \{i\}$. Set also $a_{ii}=2$.
The \emph{distinguished pre-Nichols algebra} of $V$ is  $\wtoba(V) = T(V)/\ideal(V) = \oplus_{n\in \N_0} {\wtoba}\,^{n}(V)$, where $\ideal(V)$
is the ideal of $T(V)$ generated by
\begin{itemize}
  \item relations  (41), \dots, (68) in \cite[Theorem 3.1]{Ang3},
  \item the quantum Serre relations $(\ad_c x_i)^{1-a_{ij}}x_j$ for those vertices such that $q^{a_{ij}}_{ii} = q_{ij}q_{ji}$.
\end{itemize}

The ideal $\ideal(V)$ was introduced in \cite{Ang3}, see the paragraph after Theorem 3.1; $\ideal(V)$ is a braided bi-ideal of $T(V)$,
so that there is a projection  of braided Hopf algebras $\wtoba(V) \twoheadrightarrow \toba(V)$ \cite[Proposition 3.3]{Ang3}.

\begin{definition}
We say that $p\in\{1,\ldots,\theta\}$ is a \emph{Cartan vertex} if, for every $j \neq p$, $ q_{pp}^{a_{pj}} = q_{pj}q_{jp}$.
In such case,  $\ord q_{pp}\geq 1 - a_{pj}$ by hypothesis.
\end{definition}

Clearly the projection $T(V) \to \wtoba(V)$ induces a bijection of $\bar S$ with its image $\widehat S$.
Denote again by $h$ the induced function.
Let $\widehat h: \widehat S\to \N\cup\{\infty\}$ be the function given by
$$\widehat h (s) =\begin{cases} \infty, &\text{if $s$ is conjugated to a Cartan vertex} \\
h(s) , &\text{otherwise.}
\end{cases}$$
Then the following set is a basis of $\wtoba(V)$, see the end of the proof of \cite[Theorem 3.1]{Ang3}:
\begin{align}\label{eq:pbw2}
&\big\{s_1^{e_1}\dots s_t^{e_t}: t \in \N_0, \quad s_1>\dots >s_t,\quad s_i \in \widehat S, \quad 0 < e_i<\widehat h(s_i)\big\}.
\end{align}

\subsubsection{}
The \emph{Lusztig algebra} $\lus (V)$ of $V$ is the graded dual of $\wtoba(V)$, that is $\lus (V) = \oplus_{n\in \N_0} \lus^{n}(V)$, where $\lus^{n}(V) = {\wtoba}\,^{n}(V)^*$.
The Lusztig algebra $\lus (V)$ of $V$ is the analogue of the $q$-divided powers algebra introduced in  \cite{L-jams, L-geo-ded}.

\subsection{The small quantum groups} We consider \fd{} pointed Hopf algebras  attached to the matrix $(q_{ij})_{i,j\in \I}$,
analogues of the small quantum groups or Frobenius-Lusztig kernels.
We need the following additional data: A finite abelian group $\Ga$, provided with
elements $g_1, \dots, g_\theta \in \Ga$ and characters $\chi_1, \dots, \chi_\theta \in \Hom_{\Z}(\Ga, \ku^{\times})$ such that
\begin{align}\label{eq:cartan-datum-finite}
\chi_j(g_i) &= q_{ij}, &  &i,j\in \I.
\end{align}
We define a structure of Yetter-Drinfeld module over $\ku\Ga$ on $W\oplus V$ by
\begin{align}
v_i &\in V_{g_i}^{\chi_i}, &  w_i &\in W_{g_i}^{\chi_i^{-1}}, & i&\in \I.
\end{align}
Let $\U$ be the Hopf algebra $T(W \oplus V) \# \ku\Ga / \Ig$, where $\Ig$ is the ideal generated by $\J(V)$, $\J(W)$ and the relations
\begin{align}\label{reducedlinking-finite}
&v_iw_{j} - \chi_j^{-1}(g_i) w_{j}v_i - \delta_{ij}(g_i^2 - 1) & &i,j\in \I.
\end{align}
This is a  pointed quasi-triangular Hopf algebra with $\dim \U = \vert \Ga\vert\dim \toba(V)^2$. The freedom to choose the abelian group $\Ga$ allows more flexibility, but otherwise this is very close to the small quantum groups (with more general Nichols algebras). By choosing $\Ga$ appropriately,  $\U$ is a spherical Hopf algebra.
Let $\B$ (resp. $\U^-$) be the subalgebra of $\U$ generated by $v_1,\dots,v_{\theta}$ and $\Ga$ (resp. $w_1,\dots,w_{\theta}$). Consider the morphisms of algebras $\rho_V :\toba(V) \to \U$, $\rho_W : \toba(W) \to \U$ and $\rho_{\Ga} : \ku\Ga\to \U$, given by
$\rho_V(v_i) = v_i$, $\rho_W(w_i)= w_i$,
$\rho_{\Ga}(g_i) = g_i$,  $i\in \I$. Then

\begin{enumerate}\renewcommand{\theenumi}{\alph{enumi}}   \renewcommand{\labelenumi}{(\theenumi)}
\item $\rho_V$, $\rho_W$, $\rho_{\Ga}$ give rise to isomorphisms $\toba(W) \simeq \U^-$, $\toba(V) \# \ku\Ga \simeq \B$.

\smallbreak
\item The map $ \toba(V) \otimes \toba(W) \otimes \ku\Ga \to \U$, $v \otimes w \otimes g \mapsto \rho_V(v) \rho_W(w) \rho_{\Ga}(g)$
is a coalgebra isomorphism.

\smallbreak
\item The multiplication maps $\U^- \otimes \B\to \U$ , $\B\otimes \U^-  \to \U$ are linear isomorphisms.
\end{enumerate}
See \cite[Theorem 5.2]{Ma}, \cite[Corollary 3.8]{ARS}. Now suppose that one would like to define tilting modules over $\U$,
ignoring that this is not a quasi-hereditary algebra. Inducing from $\B$, we see that simple modules correspond to characters of $\Gamma$; but the set of simple modules could not be totally ordered and \eqref{eq:pro-unique} and \eqref{eq:ind-head-simple} do not necessarily hold. A first approach to remedy this that might come to the mind is to assume the following extra hypothesis:
\emph{There exists a $\Z$-bilinear form $\langle\, , \, \rangle: \Ga\times \La \to \ku^{\times}$ such that}
\begin{align}\label{eq:cartan-datum-weights}
\langle g_i, \alpha_j\rangle &= q_{ij}, &  &i,j\in \I.
\end{align}
Then we may consider a category that is an analogue of category of representations of the algebraic group $G_1T$ in positive characteristic, or else of its quantum analogue in the literature of quantum groups.
Let $\C_{\U}$ be the category  of \fd{} $\U$-modules $M$ with a $\La$-grading $M = \oplus_{\lambda \in \La} M_{\lambda}$, compatible with the action of $\U$ in the sense
\begin{align}\label{eq:weight-space}
M_{\lambda} &= \{m\in M: g\cdot m = \langle g, \lambda\rangle m, \quad g\in \Ga\}, & \lambda \in \La, \\ \label{eq:weight-space-vi}
v_i\cdot M_{\lambda} &= M_{\lambda + \alpha_i},\qquad w_i\cdot M_{\lambda} = M_{\lambda - \alpha_i},  & \lambda \in \La,\quad i\in \I.
\end{align}
Morphisms in $\C_{\U}$ preserve both the action of $\U$ and the grading by $\La$.
The category $\C_{\B}$ is defined analogously. Both categories $\C_{\U}$ and $\C_{\B}$ are spherical tensor categories (up to an appropriate choice of $\Ga$), with duals defined in the obvious way. There are functors $\Res_{\B}^{\U}$, $\Ind_{\B}^{\U}$ and $\Pro_{\B}^{\U}$ between the categories $\C_{\U}$ and $\C_{\B}$; indeed
\begin{align*}
\Ind_{\B}^{\U} T &= \U\ot_{\B} T \simeq \toba(W) \ot T, & \Pro_{\B}^{\U} T &= \Hom_{\B} (\U, T) \simeq \Hom(\toba(W), T),
\end{align*}
so that the grading in $\Ind_{\B}^{\U} T$, resp. $\Pro_{\B}^{\U} T$, arises from that of $\toba(W) \ot T$, resp.  $\Hom(\toba(W), T)$.

\smallbreak Given $\lambda \in \La$, we denote by $\ku_{\lambda}$ the vector space with generator $\uno_{\lambda}$, considered as object in $\C_{\B}$ by
\begin{align*}
\deg \uno_{\lambda} &= \lambda, & g \cdot \uno_{\lambda} &= \langle g, \lambda\rangle \uno_{\lambda}, & g&\in \Ga, & v \cdot \uno_{\lambda} &= 0,& v&\in V.
\end{align*}
Note that $\ku_{\lambda} \simeq \ku_{\mu}$ in $\lmf{\B}$ whenever $\lambda - \mu \in \Ga^{\perp}$, but they are not isomorphic as objects in $\C_{\B}$ unless $\lambda = \mu$. Clearly, $\Irr \C_{\B} = \{\ku_{\lambda}: \lambda \in \La\}$.

\smallbreak Consider the modules $\Pro_{\B}^{\U}(\ku_{\lambda})$ and $\Delta(\lambda) := \Ind_{\B}^{\U}(\ku_{\lambda})$.
We know
\begin{align}\label{eq:pesos-delta}
\varPi(\Delta(\lambda)) &= \{\lambda - \alpha: \alpha \in \varPi(\toba(V))\},
\\ \label{eq:pesos-nabla}
\varPi(\Pro_{\B}^{\U}(\ku_{\lambda})) &= \{\lambda + \alpha: \alpha \in \varPi(\toba(V))\}.
\end{align}
Thus,  $\Delta(\lambda)$ has a highest weight  $\lambda$ and a lowest weight  $\lambda - \varrho$, both of multiplicity 1; and
$\Pro_{\B}^{\U}(\ku_{\lambda})$ has a highest weight  $\lambda + \varrho$ and a lowest weight  $\lambda$, both of multiplicity 1.
For convenience, set
$\nabla(\lambda) = \Pro_{\B}^{\U}(\ku_{\lambda - \varrho}) \simeq \Delta(-\lambda + \varrho)^*$, since $\ku_{\lambda}^* \simeq \ku_{-\lambda}$.

\smallbreak
The statements \eqref{subsub:ind-univ}, \eqref{subsub:simple-ind-quot}, \eqref{subsub:pro-univ} and \eqref{subsub:simple-pro-subobj} in \ref{subsub:pro-ind-generalities} above carry over to the present setting.
We claim that \eqref{eq:ind-unique}, \eqref{eq:pro-unique}, \eqref{eq:ind-head-simple} and \eqref{eq:pro-socle-simple} also hold here.

\smallbreak
Indeed, let $S\in \Irr \C_\U$ and $\lambda, \mu \in \varPi(S)$ such that $\mu \preceq \tau \preceq \lambda$ for all $\tau\in \varPi(S)$.
If $m \in S_{\lambda} - 0$, then $v_i\cdot m =0$ for all $i\in \I$ by \eqref{eq:weight-space-vi}, hence $\ku m \simeq \ku_{\lambda}$
and we have $\Delta(\lambda) \twoheadrightarrow S$ and $\varPi(S) \subseteq \{\lambda - \alpha: \alpha \in \varPi(\toba(V))\}$
by \eqref{eq:pesos-delta}. Moreover, if $\Delta(\lambda') \twoheadrightarrow S$, then
$\lambda \in \varPi(S) \subseteq \{\lambda' - \alpha: \alpha \in \varPi(\toba(V))\}$, hence $\lambda' = \lambda$,
showing \eqref{eq:ind-unique}. The proof of \eqref{eq:ind-head-simple} is standard: $\Delta(\lambda)$ has a unique maximal submodule,
which is the sum of all submodules intersecting trivially $\Delta(\lambda)_{\lambda}$. Now \eqref{eq:pro-unique}
and \eqref{eq:pro-socle-simple} follow by duality, so that $S \hookrightarrow \nabla(\mu)$.

In conclusion we have the following standard result.

\begin{prop}\label{prop:simple-modules} If $E(\lambda) := $ head of $\Delta (\lambda)$, then $\Irr \C_{\U} = \{E(\lambda): \lambda \in \La\}$. If $\mu\in \Ga^{\perp}$, then $\dim E(\mu)=1$ and $E(\mu) \ot E(\lambda) \simeq E(\lambda) \ot E(\mu)\simeq E(\lambda + \mu)$. There is a bijection $w_0: \La \to \La$ such that $E(\lambda) := $ socle of $\nabla (w_0(\lambda))$. \qed
\end{prop}

The modules $\Delta(\lambda)$, resp. $\nabla(\lambda)$, are called the Weyl modules, resp. the dual Weyl modules. We may then go on and define good and Weyl
filtrations, and tilting modules.
However, it is likely that tilting modules are projective, thus with 0 quantum dimension, as is the case for $G_1T$, see \cite[3.4]{A-quot}, \cite{Jantzen}.

\subsection{Generalized quantum groups}
The next idea is to replace $\Ga$ by an infinite abelian group $Q$, perhaps free of finite rank, and the
Nichols algebras $\toba(V)$, $\toba(W)$ by the distinguished pre-Nichols algebras $\wtoba(V)$, $\wtoba(W)$.
Namely, we assume that $Q$ is provided with
elements $K_1, \dots, K_\theta$ and characters $\Upsilon_1, \dots, \Upsilon_\theta \in \Hom_{\Z}(Q, \ku^{\times})$ such that
$\Upsilon_j(K_i) = q_{ij}$, $i,j\in \I$.
Then $W\oplus V$ is also a Yetter-Drinfeld module over $\ku Q$ by
$v_i \in V_{K_i}^{\Upsilon_i}$, $w_i \in W_{K_i}^{\Upsilon_i^{-1}}$, $i\in \I$.
Let $\Udcp(V) = T(W \oplus V) \# \ku Q / \widehat{\Ig}$ where $\widehat{\Ig}$ is the ideal generated by $\ideal(V)$, $\ideal(W)$ and the relations
\begin{align}\label{reducedlinking}
&v_iw_{j} - \chi_j^{-1}(g_i) w_{j}v_i - \delta_{ij}(g_i^2 - 1) & &i,j\in \I.
\end{align}
This Hopf algebra, for a suitable $Q$, was introduced in \cite{Ang3}; it
 is the analogue of the quantized enveloping algebra at a root of one for $(q_{ij})_{i,j\in \I}$ in the version of \cite{dCP}.
It also has a triangular decomposition similar as in the case of $\U$. Furthermore, there are
so-called Lusztig isomorphisms, because they generalize the braid group representations defined by Lusztig, see e.~g. \cite{L}.
Actually, the definition of the ideal $\widehat{\Ig}$ in \cite{Ang3} was designed to have (a) a braided bi-ideal, and (b) the Lusztig automorphisms at the level of $\Udcp(V)$, generalizing results from \cite{H-isom}. More precisely, the situation is as follows.

\smallbreak We assume that $Q$ and $\Ga$ are accurately chosen and that there is a group epimorphism $Q\to \Ga$.
Given $i\in \I$, we define the $i$-th reflection of $(V,c)$.  Define $s_i\in\Aut \La$ by $s_i(\alpha_j)=\alpha_j-a_{ij}\alpha_i$, see \eqref{defn:mij}.
Then $s_i(V, c) = (V, c')$, where
$C'$ is the braiding of diagonal type with matrix $(\widetilde{q}_{rs})_{r,s\in \I}$. Here $\widetilde{q}_{rs}= (s_i(\alpha_r)\vert s_i(\alpha_s))$; we omit the mention to the braidings $c$, $c'$, etc.
Then

\begin{itemize}\renewcommand{\labelitemi}{$\circ$}
  \item There are algebra isomorphisms $T_i, T_i^-: \U(V) \to \U(s_iV)$,
such that $T_i T_i^-=T_i^- T_i = \id_{\U(V)}$ \cite[Theorem 6.12]{H-isom}.

  \item There are algebra isomorphisms $T_i, T_i^-: \Udcp(V) \to \Udcp(s_iV)$,
such that $T_i T_i^-=T_i^- T_i = \id_{\Udcp(V)}$ \cite[Proposition 3.26]{Ang3}, compatible with those of $\U(V)$.
\end{itemize}

There is a $\La$-grading on $T(W \oplus V)\# \ku Q$ given by $\deg \gamma = 0$, $\gamma\in Q$,
$\deg v_i = \alpha_i = -\deg w_i$, $i\in \I$; it extends to gradings of $\U(V)$ and $\Udcp(V)$.
Hence we may consider categories $\C_{\Udcp}$ and so on.
However, the Hopf algebra $\Udcp(V)$ has a large center $Z$ and is actually finite over it. Thus, it seems that
its representation theory should be addressed with the methods of \cite{dCP, dCPRR}.

It remains a third tentative: to repeat the above considerations replacing the distinguished
pre-Nichols algebras $\wtoba(V)$, $\wtoba(W)$ by the Lusztig algebras $\lus(V)$, $\lus(W)$.
We hope to address this in future publications.

\bigbreak
\subsection*{Acknowledgements} We thank Olivier Mathieu for a course on tilting modules given at the University of
C\'ordoba in November 2009; and also Alain Brugui\`eres, Flavio Coelho, Juan Cuadra,  Michael M\"uger, Dmitri Nikschych, Claus Ringel and Yorck Sommerh\"auser
for very useful conversations.

\end{document}